\documentclass[11pt]{amsart}
\usepackage{amsfonts, bm}
\usepackage{mathrsfs}
\allowdisplaybreaks
\usepackage{hyperref}
\usepackage{color}
\usepackage{cite,enumerate}
\makeatletter

\newtheorem{theorem}{Theorem}[section]
\newtheorem{lemma}[theorem]{Lemma}

\newtheorem{proposition}[theorem]{Proposition}
\numberwithin{equation}{section}

\allowdisplaybreaks \numberwithin{equation}{section}

\begin{document}
\title[Semi-wave and sharp estimates of propagation ]{Semi-wave and sharp estimates of propagation  for monostable free boundary problems  in time-periodic environment}
\author[ Y. Du and  Z. Ma ]{ Yihong Du$^{\dag}$ and Zhuo Ma$^{\dag}$
}
 \thanks{
 \mbox{$^{\dag}$ School of Science and Technology, University of New England, Armidale, NSW 2351, Australia.} 
 \\
 %\mbox{\ \ \ \ \  $^{\ddag}$ School of Mathematics and Statistics, Lanzhou University,
 %Lanzhou, Gansu 730000, China.} \\
\mbox{\ \ \ \ \ \ \    Emails:}  ydu@turing.une.edu.au (Y. Du),\ mazh25@foxmail.com (Z. Ma).}
 
\date{\today}
\maketitle 
\begin{abstract}
We investigate the propagation profile of positive solutions to 
  \begin{equation*}
    u_t-du_{xx}=f(t,u) \mbox{ for }  t>0,\ x\in(g(t),h(t)),
  \end{equation*}
  where $f(t,u)$ is  monostable in $u$ and $T$-periodic in $t$, and the free boundaries $x=g(t), \ x=h(t)$ are determined by the Stefan condition $g'(t)=-\mu u_x(t, g(t)),\ h'(t)=-\mu u_x(t,h(t))$, coupled with $u(t, g(t))=u(t, h(t))=0$. 
  For a special nonlinearity satisfying the strong KPP condition, the  long-time behavior and asymptotic spreading speed  of this problem were considered by Du, Guo and Peng \cite{DGP}.  In this paper, by employing new techniques,  we extend the results of \cite{DGP} to general monostable nonlinearities beyond the KPP framework and at the same time we obtain more precise description of the propagation profile: we prove the existence and uniqueness of a semi-wave and show that the spreading solution converges to this semi-wave as time goes to infinity. 
  %Our results suggest that, compared with the KPP case, general monostable nonlinearities may lead to a larger asymptotic spreading speed for the free boundary problem under consideration.

\bigskip



\vspace{1cm}

\noindent
\underline{\it Dedicated to Professor Hiroshi Matano for his 75th birthday.}
\end{abstract}

\section{Introduction and main results}
In this paper, we obtain sharp estimate for the propagation profile determined by the following free boundary model with a  time-periodic monostable nonlinearity:
 \begin{equation}\label{free-bound}
\begin{cases}
  u_t-du_{xx}=f(t,u), &t>0,~g(t)<x<h(t),\\
  u(t,g(t))=0, u(t,h(t))=0, &t>0,\\
  g'(t)=-\mu u_x(t,g(t)),&t>0,\\
  h'(t)=-\mu u_x(t,h(t)),&t>0,\\
  -g(0)=h(0)=h_0, u(0,x)=u_0(x), &-h_0\leq x\leq h_0,
  \end{cases}
\end{equation}
where $d$, $\mu$ and $h_0$ are given positive constants;  $x=g(t)$ and $x=h(t)$ are  moving boundaries to be determined together with  $u(t,x)$. The nonlinear function $f:\mathbb{R}\times[0,\infty)\to \mathbb{R}$ is assumed to satisfy the following smoothness conditions:
\begin{equation}\label{smooth}
\begin{cases}
\mbox{ $f(t, u)$ is  of class $C^{\alpha/2}$ in $t$ uniformly with respect to  $u\geq 0$  for some $\alpha\in(0,1)$,}\\
\mbox{ and is of class $C^1$ in $u$  uniformly with respect  to $t\in\mathbb{R}$. }
\end{cases}
\end{equation}
Apart from the above smoothness requirement, we assume that  $f(t,u)$ is periodic in $t$ and monostable in $u$, namely:
\begin{enumerate}
  \item[($\rm{\bf{f_p}}$):] (Time-periodic) There exists $T>0$ such that $f(t+T,\cdot)=f(t, \cdot)$ for $t\in\mathbb{R}$;
  \item[($\rm{\bf{f_m}}$):](Monostable) For every fixed $t\in\mathbb R$, $f(t, 0)=f(t,1)=0$, $f(t,u)>0 $ for $u\in(0,1)$, $f(t, u)<0 $ for $u\in(1,+\infty)$, and
  \begin{equation}\label{eq-f01}
  \int_0^Tf_u(t, 0)dt>0>\int_0^Tf_u(t,1)dt.
  \end{equation}
  \end{enumerate}
  The initial function $u_0(x)$ is assumed to be an element of $I(h_0)$ with
 \begin{equation*}
  I(h_0):=\{u_0\in C^2([-h_0,h_0]):u_0(-h_0)=0=u_0(h_0), u_0(x)>0\text{ in }(-h_0,h_0)\}.
 \end{equation*}
  
  In the special case that $f(t,u)=f(u)$ is independent of $t$, \eqref{free-bound} reduces to a model considered in \cite{DL10, DL15, DLZ15}, which can be regarded as a free boundary version of the corresponding Cauchy problem 
  \begin{equation}\label{Cauchy}
  \begin{cases}
  u_t-du_{xx}=f(u), &t>0,~ x\in\mathbb R,\\
  u(0,x)=\tilde u_0(x), &x\in\mathbb R,
  \end{cases}
\end{equation}
  where $\tilde u_0(x)$ is the zero extension of $u_0(x)$ from $[-h_0, h_0]$ to $\mathbb R$. Problem \eqref{Cauchy} has been used as a
  model for propagation of a species with density $u(t,x)$ at time $t$ and spatial location $x$ since the pioneering works of Fisher \cite{Fisher} and  Kolmogorov-Petrovskii-Piskunov (KPP) \cite{KPP}, and starting from Aronson-Weinberger \cite{AW, AW78}, significant further developments have been achieved along different lines of extensions (see, for example, \cite{BH, BHM, fife77, HNRR, LZ, shen10, W82, W02, Xin}). Compared with the classical model \eqref{Cauchy}, the free boundary model \eqref{free-bound} has the advantage that the population range is explicitly given in the model by $[g(t), h(t)]$, with the free boundaries $x=g(t)$ and $x=h(t)$ representing the spreading fronts. The free boundary condition in \eqref{free-bound} coincides with the Stefan condition used to describe the melting of ice in contact with water, 
 a derivation of this condition based on some ecological assumptions can be found in \cite{BDK}. On the other hand, when the parameter $\mu$ in \eqref{free-bound} converges to $\infty$, it is known that \eqref{Cauchy} is the limiting problem of \eqref{free-bound} (see \cite{DG}).
 
 For \eqref{free-bound} in the homogeneous case with $f(u)$ monostable (as well as bistable or of combustion type), the precise spreading profile  was obtained in \cite{DMZ} (see also \cite{DMZ2} for the high space dimension case), which indicates that, when spreading happens, for large time, the solution converges to
 the associated semi-wave solution. This result reveals a sharp difference from the classical  model \eqref{Cauchy} in that, with a monostable $f(u)$, it has a minimal wave speed $c_*>0$ so that \eqref{Cauchy} has a traveling wave solution if and only if its speed $c\geq c_*$, and  the solution of \eqref{Cauchy} is approximated by the traveling wave solution with minimal speed $c_*$, but the precise approximation involves  a logarithmic correction term when $f(u)$ satisfies additionally the KPP condition\footnote{Namely, $f(u)\leq f'(0)u$ for $u\in [0,1]$.}; such a correction term does not occur for the homogeneous \eqref{free-bound}, where the density function $u(t,x)$ is approximated by the unique semi-wave solution without any logarithmic shift (resembling the case of \eqref{Cauchy} with a bistable $f(u)$ considered in  \cite{fife77}). We refer to \cite{AGX, AHR, Bramson, HNRR} and the references therein for some results on the precise propagation profile of \eqref{Cauchy}.

However, the understanding of \eqref{free-bound} in the time-periodic case has not reached the level described in the previous paragraph in the existing works.  { It is easy to extend the proofs of \cite{DGP} to show that \eqref{free-bound} admits a unique global solution triple $(u(t,x),g(t),h(t))$. Moreover, if $f_u(\cdot,0)>0$ ,  its long-time behavior exhibits a spreading-vanishing dichotomy,} namely one of the following must happen:
\begin{enumerate}
  \item [(i)] \textbf{spreading}: as $t\to\infty$, $(g(t),h(t))$ converges to $ (-\infty,\infty)$, and $u(t,x)$ converges locally uniformly to the unique positive $T$-periodic solution of $w'=f(t,w)$ in $\mathbb R$,
  \item [(ii)]\textbf{vanishing}:  as $t\to\infty$, $(g(t),h(t))$ converges to some finite interval $ (g_\infty,h_\infty)$, and $u(t,x)$ converges uniformly to $0$ as $t\to\infty$.
\end{enumerate}

In the special case $f(t,u)=u[a(t)-b(t)u]$ with $a(t)$ and $b(t)$ positive $T$-periodic H\"older continuous functions, it was shown in \cite{DGP} that \eqref{free-bound} has a unique semi-wave solution, which determines the spreading speed of a high dimension version of \eqref{free-bound} with radial symmetry, but a sharp approximation of the solution by the semi-wave profile is lacking. Moreover, the proof there relied substancially on the special properties of the particular form of $f(t,u)$; more precisely the strong KPP property that $f(t,u)/u$ is decreasing in $u>0$, has played an essential role in the proof.

It is expected that the results in \cite{DGP} remain valid for a general monostable $f(t,u)$, and moreover, the estimate of the propagation profile can be sharppened as in the homogeneous case of \cite{DMZ}. The main purpose of this paper is to show that these indeed can be achieved. This is largely inspired by our recent work \cite{DMW25},  where similar results have been proved for a related but different free boundary model in time-periodic environment by using new techniques -- it turns out that  some of the techniques there can be adapted to treat \eqref{free-bound}
here, although additional new techniques are also required.
%%%%%%%%%%%%%%%%
 
\medskip

Our main results for \eqref{free-bound} are the following two theorems.

 \begin{theorem}\label{theo-semi-wave} Suppose that  \eqref{smooth}, $\rm{(\bf{f_p})}$  and $\rm{(\bf{f_m})}$ hold.  
Then given any $\mu>0$, the equations
  \begin{equation}\label{eqn-wave}
   \begin{cases}
     \Psi_t-d\Psi_{xx}+k(t)\Psi_x=f(t,\Psi), \ { 0\leq \Psi\leq 1}, &t\in[0,T],~x\in(0,\infty),\\
\Psi(t,0)=0, &t\in[0,T],\\
\Psi(0,x)=\Psi(T,x), &x\in[0,\infty)\\
   \end{cases}
  \end{equation}
and 
  \begin{equation}\label{eq-phix1}
    \mu \Psi_x(t,0)=k(t)\quad  \text{ for }t\in[0,T]
  \end{equation}
 admit a unique  solution pair $(k(t),\Psi(t,x))$  with $k(t)$  a H\"older  continuous, positive $T$-periodic function. Furthermore, $\Psi(t,\infty)=1$, $\Psi_x(t,x)>0$ for $t\in[0,T]$ and  $x\in[0,\infty)$, and there exists a constant $c^*$ independent of $\mu$ such that $\bar{k}:=\frac{1}{T}\int_0^Tk(s)ds<c^*$. 
\end{theorem}

\begin{theorem}\label{thm-exact} Suppose  that  \eqref{smooth}, $\rm{(\bf{f_p})}$ and  $\rm{(\bf{f_m})}$ hold. Let $(k,\Psi)$ be the semi-wave given by Theorem {\rm\ref{theo-semi-wave}} and $(u,g,h)$ be  a spreading  solution of \eqref{free-bound} with $u_0\in I(h_0) $. Then there exists $g^*, h^*\in\mathbb{R}$ such that 
  \[
  \begin{cases}
  \lim_{t\to\infty}\left[g(t)+\int_0^tk(s)ds\right]=g^*,~~\lim_{t\to\infty}[g'(t)+k(t)]=0,\\[2mm]
   \lim_{t\to\infty}\left[h(t)-\int_0^tk(s)ds\right]=h^*,~~\lim_{t\to\infty}[h'(t)-k(t)]=0,\\
   \end{cases}
   \]
   \[\begin{cases}
    \lim_{t\to\infty} \sup_{x\in[g(t), 0]}|u(t,x)-\Psi(t,x-g(t))|=0,\\[2mm]
    \lim_{t\to\infty} \sup_{x\in[0, h(t)]}|u(t,x)-\Psi(t,h(t)-x)|=0.\\ 
  \end{cases}
  \]
 \end{theorem}

As stated earlier, Theorem \ref{theo-semi-wave} was proved previously in \cite{DGP} for the special case $f(t,u)=u[a(t)-b(t)u]$. In that special case, the integral average of the wave speed $k(t)$ was shown to be less than $2\sqrt{d \bar{a}}$, whereas the constant $c^*$ appearing in Theorem \ref{theo-semi-wave}  lies between $2\sqrt{d \bar{a}_0}$ and $2\sqrt{d \bar{a}_1}$ (see \eqref{eq-c} and Lemma \ref{lem-ex-unex} below), where 
\[
\ a_0(t):=f_u(t,0),\ \ a_1(t):=\sup_{0<u<1}\frac{f(t,u)}{u},\]
  and for any $T-$periodic function $\alpha(t)$, its average is denoted by $\bar \alpha$; namely,
  \[
   \bar \alpha:=\frac 1T\int_0^T\alpha(t)dt.
\]
 So the characterization of $c^*$ is no longer determined by $f_u(t,0)$ in general without  the KPP condition. 

 Theorem \ref{thm-exact} gives the precise  asymptotic profile for $(u(t,x), g(t), h(t))$, to a level only achieved in previous works with homogeneous environment; see 
  \cite{DMZ,DMZ2}. 
 This is a significant improvement on the results of \cite{DGP} in time-periodic environment.
  We believe the method in this paper can   be further developed to treat the high-space dimension version of \eqref{free-bound} with radial symmetry, which will be considered in future work.
 
 %Theorems \ref{theo-semi-wave} and \ref{thm-exact} show that, although the spreading speed may increase in the absence of the KPP condition on $f$, the manner in which $u$ converges toward the semi-wave $\Phi$ remains similar.

 The rest of the paper is organized as follows.  Section~2 focuses on the existence, uniqueness and asymptotic estimate of the semi-wave solution determined by \eqref{eqn-wave} and \eqref{eq-phix1}. To obtain our existence result for the system \eqref{eqn-wave} and \eqref{eq-phix1}, it is necessary to first obtain some conditions for the existence as well as the nonexistence of positive solutions to the first equation \eqref{eqn-wave}, which requires techniques well beyond those in \cite{DMW25} and other existing works. In Section 3,  we obtain the asymptotic spreading speed and the precise  propagation profile for spreading solutions, by constructing suitable upper and lower solutions, followed by  other delicate analysis involving parabolic regularity arguments and completely solving certain limiting problems.  
 
 %To overcome the difficulties arising from the weaker condition \eqref{eq-f01}, several nontrivial modifications are introduced.

We end the introduction by mentioning some further related works on free boundary models. We refer to \cite{SLZ,LLS,W16} for more results in time-periodic media, and to \cite{DDL19,DL} and the references therein for results with spatially heterogeneous environment.
There are many related works for more general nonlinearities, in high space dimensions, or in shifting environments, or with time delay;  as a very small sample, we mention \cite{DG, DHL, DMW, KMY, DFS}, where further references can be found.

\section{Existence, uniqueness and asymptotic behavior of semi-wave}
In this section, we investigate the existence, uniqueness and qualitative properties of the semi-wave. 
To prove the  existence and uniqueness result, we begin by analyzing some associated eigenvalue problems, which enables us to define a critical value $c^*$ for the integral average of $k(t)$ as a sharp upper bound: problem \eqref{eqn-wave} admits a positive solution  if and only if $\bar{k}< c^*$.   Building on this result, we are able to further develop the arguments in \cite{DGP} and \cite{DMW25} to prove the existence of a unique positive function $k(t)$ such that the associated positive solution of \eqref{eqn-wave} satisfies \eqref{eq-phix1}. 

We would like to point out that in the KPP case (such as for the special $f(t,u)$ considered in \cite{DGP}), the critical value $c^*$ is linearly determined, which makes the existence question much easier to handle technically. For the  general monostable case here, $c^*$ is in general not linearly determined, and we need to introduce some  nontrivial techniques to treat the existence part of the problem (see Lemmas \ref{lem-ex-unex}, \ref{lem-suff},  \ref{TW-SW-0}, \ref{TW-SW} and \ref{lem-iff} below). Note also that such a critical value $c^*$ does not arise for the model of \cite{DMW25}, where the corresponding problem of \eqref{eqn-wave} has a positive solution for every $k(t)$.
\medskip

We will need the following  special comparison principle  in \cite{MDW}:
\begin{lemma}{\rm\cite[Lemma 3.2]{MDW}}\label{lem-maxi}
Assume that  $k(t)$ is a $T$-periodic nonnegative continuous function. Let $u, v\in C([0,T]\times[0,L])\cap C^{1,2}([0,T]\times(0,L))$ be two $T$-periodic bounded functions such that 
  \begin{equation}\label{eq-maxuv}
    \begin{cases}
    u_t-d u_{xx}+ k(t)u_x -f(t,  u)\geq0\geq v_t-d v_{xx}+k(t)v_x -f(t,  v),&t\in[0,T],x\in(0,L), \\
    u(t,0)\geq v(t,0),u(t,L)\geq v(t,L), &t\in[0,T],\\
 u(t,0)<u(t,x), &t\in [0,T],x\in(0,L],\\
 v(t, L)>v(t,x),   &t\in [0,T],x\in[0,L).
 \end{cases}
\end{equation}
 Then  $u\geq v$ in $[0,T]\times (0,L)$. Moreover, if $u(t,L)\not\equiv v(t,L)$, then $u>v$ in $[0,T]\times (0,L)$. 
  \end{lemma}

Write 
  \begin{equation}\label{eq-ka01}
    \tilde k(t):=k(t)-\bar{k}, \ \ a_0(t):=f_u(t,0),\ \ a_1(t):=\sup_{0<u<1}\frac{f(t,u)}{u}.
 \end{equation}
 Clearly, $a_0(t)\leq a_1(t)$ for all $t\in\mathbb R$.
 Denote 
\begin{equation*}
  \begin{cases}
       p_i(t):=\frac{\bar{k}^2}{4d}+\frac{\bar{k}}{2d}\tilde{k}(t)-a_i(t), \\
      \mathcal{L}_i:=\partial_t-d\partial_{xx}+\tilde k(t)\partial_x+p_i(t),
  \end{cases}i=0,1.
\end{equation*}
We then consider the following  eigenvalue problem:
  \begin{equation}\label{eq-eig}
\begin{cases}
    \mathcal{L}_0\phi=\lambda\phi, &(t,x)\in[0,T]\times \mathbb R,\\
    \phi(0,x)=\phi(T,x), & x\in\mathbb R. \\
\end{cases}
  \end{equation}
Since $\int_0^T\tilde k(t)dt=0$, it follows from  \cite[Theorems 2.7, 2.13, Proposition 2.14]{Nadin}(see also \cite[Proposition 2.2]{DGP}) that \eqref{eq-eig}  admits a principal eigenvalue $\Lambda_0$ and the associated principal eigenfunction  is positive and independent of $x$, denoted by $\phi_0(t)$. The pair $(\Lambda_0,\phi_0(t))$ satisfies
\begin{equation*}
   \phi_0'+p_0(t)\phi_0=\Lambda_0\phi_0\ \mbox{ for } t\in[0,T]; \ \ \phi_0(0)=\phi_0(T).\\
\end{equation*}
It follows easily that
\[\Lambda_0=\bar{p}_0=\frac{\bar{k}^2}{4d}-\bar{a}_0.\]
Moreover, by Propositions 2.3 of \cite{Nadin}, the principal eigenvalue $\lambda_1^m$ of 
  \begin{equation*}
\begin{cases}
    \mathcal{L}_0\phi=\lambda\phi, &(t,x)\in[0,T]\times [-m,m],\\
    \phi(t,\pm m)=0, & t\in[0,T], \\
    \phi(0,x)=\phi(T,x), & x\in[-m,m] \\
\end{cases}
  \end{equation*}
  satisfies
  \begin{equation}\label{eq-lam-lim}
   \lim_{m\to\infty} \lambda_1^m=\Lambda_0. 
  \end{equation}

Similarly,  the eigenvalue problem
  \begin{equation}\label{eq-eig3}
\begin{cases}
    \mathcal{L}_1\phi=\lambda\phi, &(t,x)\in[0,T]\times \mathbb R,\\
    \phi(0,x)=\phi(T,x), & x\in\mathbb R
\end{cases}
  \end{equation}
admits a principal eigenvalue pair $(\Lambda_1,\phi_1(t))$  and 
\[\Lambda_1=\bar{p}_1=\frac{\bar{k}^2}{4d}-\bar{a}_1.\]

We now investigate the existence and nonexistence of positive solutions to \eqref{eqn-wave}. For $i=0,1$, to stress the dependence of $\Lambda_i$ on $k(t)$, we will from now on write $\Lambda_i=\Lambda_i^k$.
\begin{lemma}\label{lem-ex-unex}
Given any nonnegative $T$-periodic function $k(t)\in C^{\alpha/2}([0,T])$,  problem \eqref{eqn-wave} admits a positive solution $\Psi^k\in C^{1,2}([0,T]\times[0,\infty))$   if $\bar k<2\sqrt{\bar a_0 d}$, and  \eqref{eqn-wave} has no positive solution lying between 0 and 1 if $\bar k\geq 2\sqrt{\bar a_1 d}$. Moreover, any positive solution $\Psi^k$ of \eqref{eqn-wave} lying between 0 and 1 satisfies 
\begin{equation}\label{prop}
 \Psi^k_x(t,x)>0, \ \Psi^k(t,\infty)=1\ \mbox{ for } (t,x)\in[0,T]\times(0,\infty). 
\end{equation}
\end{lemma} 

\begin{proof}
We first show that problem \eqref{eqn-wave}  admits a positive solution when $\bar k<2\sqrt{\bar a_0 d}$, i.e., when $\Lambda_0^k<0$. Suppose $\Lambda_0^k<0$; by \eqref{eq-lam-lim}, $ \lambda_1^m<0$ for all large $m$. Fix such a large $m$, and let $\phi^m(t,x)$ be an associated positive eigenfunction of $ \lambda_1^m$. Then 
\[\psi^m(t,x):=e^{\frac{\bar{k}}{2d}x}\phi^m(t,x-m)\]
satisfies
 \begin{equation}\label{eq-eig2}
\begin{cases}
   \psi^m_t-d\psi^m_{xx}+k(t)\psi^m_x-a_0(t)\psi^m=\lambda_1^m\psi^m, &(t,x)\in[0,T]\times (0,2m),\\
    \psi^m(t, 0)=  \psi^m(t, 2m)=0, & t\in[0,T], \\
    \psi^m(0,x)=\psi^m(T,x), & x\in[0,2m]. \\
\end{cases}
  \end{equation}
 Since $f(t,u)$ is $ C^{1}$ in $u$ uniformly with respect to $t\in\mathbb R$, there exists  $\delta>0$ small enough such that  
\begin{equation*}
    f(t,u)\geq (f_u(t,0)+\lambda_1^m)u \ \ \text{ for }t\in[0,T], u\in[0,\delta].
\end{equation*}
Let $\epsilon>0$ be a small number such that $\epsilon\psi^m(t,x)<\delta$ for $(t,x)\in[0,T]\times[0,2m]$. Define
\begin{equation*}
 \bar{\Psi}(t,x)\equiv 1, \quad \quad 
      \underline{\Psi}(t,x):=
  \begin{cases}
\epsilon\psi^m, &(t,x)\in[0,T]\times[0,2m],\\
0, &(t,x)\in[0,T]\times(2m, \infty).\\
  \end{cases}
\end{equation*}
Then for any given $L>2m$, $\bar{\Psi}$ and  $\underline{\Psi}$ form a pair of upper and lower solutions of 
  \begin{equation}\label{eqn-finitewave}
   \begin{cases}
     v_t-dv_{xx}+k(t)v_x=f(t,v),&t\in[0,T],~x\in(0,L),\\
v(t,0)=0,~v(t,L)=1,&t\in[0,T],\\
v(0,x)=v(T,x),  &x\in[0,L].\\
   \end{cases}
  \end{equation}
By the standard upper and lower solution arguments, we know that \eqref{eqn-finitewave} admits a solution $\Psi_L$ satisfying 
\begin{center}
$ \underline{\Psi}(t,x)\leq \Psi_L(t,x)\leq1$ for $(t,x)\in[0,T]\times[0,L]$.
\end{center} 
The strong maximum principle implies  that  $ 0<\Psi_L<1$. This allows us to employ Lemma \ref{lem-maxi} to deduce that $ \Psi_L$ is the uniqueness solution of \eqref{eqn-finitewave} lying between 0 and 1, which in turn  can be used to show that $\Psi_L$ is  decreasing in $L$. Therefore, the  limit
$$\Psi^k(t,x):=\lim_{L\to\infty}\Psi_L(t,x) \text{ for } (t,x)\in[0,T]\times[0,\infty)$$ 
is well-defined,  it satisfies $\underline{\Psi} \leq \Psi^k\leq 1$ and solves \eqref{eqn-wave}.

 By the parabolic strong maximum principle and the Hopf boundary lemma we further have
\begin{center}
$0<\Psi^k(t,x)<1$, $\Psi^k_x(t,0)>0$ \  for $(t,x)\in[0,T]\times(0,\infty)$.
\end{center}
Thus   $\Psi^k$ is a positive solution of \eqref{eqn-wave}. 
Moreover, 
by employing the moving-plane method and some limiting arguments (see, e.g., \cite[Proposition~2.1]{DGP}), we can further show that \eqref{prop} holds. 

Clearly the above arguments involving the strong maximum principle and the Hopf boundary lemma apply to any positive solution of \eqref{eqn-wave}, therefore we can use the arguments in \cite{DGP} to show that \eqref{prop} holds for any positive solution of \eqref{eqn-wave} lying between 0 and 1.

Next we  show that \eqref{eqn-wave} has no positive solution lying between 0 and 1  if $\bar k\geq 2\sqrt{\bar a_1 d}$, i.e., if $\Lambda_1^k\geq 0$.  Let $(\Lambda_1,\phi_1(t))=(\Lambda_1^k, \phi_1^k(t))$ be a principal eigenvalue pair of \eqref{eq-eig3}.
Define
\begin{equation*}
  \psi_1(t,x):=e^{\frac{\bar{k}}{2d}x}\phi_1(t).
\end{equation*}
Then 
\begin{equation*}
  (\psi_1)_{t}-d (\psi_1)_{xx}+k(t) (\psi_1)_x-a_1(t)\psi_1=\Lambda_1\psi_1 \text{ for }(t,x)\in[0,T]\times\mathbb R.
\end{equation*}
Since $\Lambda_1\geq 0$,  one has $ \bar{k}\geq 2\sqrt{d\bar{a}_1}>0$, which implies
\begin{equation*}
  \psi_1(t,x)\geq \min_{t\in[0,T]}\phi_1(t)>0 \ \ \text{ for }(t,x)\in[0,T]\times[0,\infty),
\end{equation*}
and moreover,
\begin{equation}\label{eq-psioo}
  \psi_1(t,x)\to\infty \text{ as }x\to\infty\text{  uniformly in } t\in[0,T].
\end{equation}

Suppose $\Psi$ is a positive solution of \eqref{eqn-wave} satisfying $0\leq \Psi\leq 1$. Define 
\begin{equation*}
  \tau_{min}:=\min\{\tau\in(0,\infty):\tau\psi_1(t,x)\geq \Psi(t,x) \text{ for }(t,x)\in[0,T]\times[0,\infty)\}.
\end{equation*}
Clearly,  $\tau_{min}\geq 0$ is well-defined and 
\begin{equation*}
\tau_{min}\psi_1(t,x)\geq \Psi(t,x) \text{ for }(t,x)\in[0,T]\times[0,\infty).
\end{equation*} 
We claim that $\tau_{min}=0$.  Otherwise $\tau_{min}>0$, and by \eqref{eq-psioo},  there exists $x_0>0$ large enough and $\epsilon_0$ small enough such that for every $\epsilon\in[0,\epsilon_0]$,
\begin{eqnarray}\label{eq-x_0tau}
  (\tau_{min}-\epsilon)\psi_1(t,x)\geq1\geq  \Psi(t,x) \text{ for }(t,x)\in[0,T]\times[x_0,\infty).
\end{eqnarray}
Moreover, it follows from the definition of $a_1(t)$ that 
\begin{equation*}
 \Psi_t-d\Psi_{xx}+k(t)\Psi_{x}-a_1(t)\Psi\leq \Psi_t-d\Psi_{xx}+k(t)\Psi_{x}-f(t,\Psi)=0.
\end{equation*}
Due to $\Lambda_1\geq 0$, we see that  $W:=\tau_{min}\psi_1-\Psi\geq 0$ satisfies
\begin{equation*}
  \begin{cases}
    W_t-dW_{xx}+k(t)W_x-a_1(t)W\geq 0, &(t,x)\in(0,T)\times(0,x_0),\\
    W(t,0)> 0, \ \ W(t,x_0)>0, &t\in(0,T),\\
    W(0,x)=W(T,x),\ \ &x\in(0,x_0).
  \end{cases}
\end{equation*}
It then follows from the parabolic strong maximum principle that $W(t,x)>0$ in $[0,T]\times [0, x_0]$ and hence there exists $\epsilon_1<\epsilon_0$ such that for every $\epsilon\in[0,\epsilon_1]$,
\begin{equation*}
  W(t,x)>\epsilon\psi_1(t,x) \ \text{ for }(t,x)\in[0,T]\times[0,x_0].
\end{equation*}
Combining this with \eqref{eq-x_0tau}, we have for all $\epsilon\in(0,\epsilon_1)$, 
\begin{equation*}
    (\tau_{min}-\epsilon)\psi_1(t,x)\geq \Psi(t,x) \text{ for }(t,x)\in[0,T]\times[0,\infty),
\end{equation*}
which contradicts the definition of $\tau_{min}$, and hence $\tau_{min}=0$.  It follows that  $\Psi\equiv 0$, a contradiction to $\Psi(t,\infty)=1$. This completes the proof.
\end{proof}

Let $C_p^{\alpha/2}([0,T])$ denote the space of all $T$-periodic functions in $C^{\alpha/2}([0,T])$.
Define
\begin{equation}\label{eq-c}
  c^*:=\sup\{\bar{k}: k\in C_p^{\alpha/2}([0,T]) \text{ is nonnegative and such that \eqref{eqn-wave} has a positive solution} \}.
\end{equation}
  By Lemma \ref{lem-ex-unex}, the constant $c^*$ is well-defined and $2\sqrt{d\bar{a}_0}\leq c^*\leq 2\sqrt{d\bar{a}_1}$. By the definition of $c^*$, problem \eqref{eqn-wave} admits no positive solution when $\bar{k}>c^*$. Next, we show that for every nonnegative function $k\in C_p^{\alpha/2}([0,T])$ satisfying $\bar{k}<c^*$, problem \eqref{eqn-wave} has a positive solution.

\begin{lemma}\label{lem-suff}
Given any nonnegative function $k\in C_p^{\alpha/2}([0,T])$, problem \eqref{eqn-wave} admits a positive solution $\Psi$ satisfying $\Psi_x>0$ and $\Psi(t,\infty)=1$ if  $\bar{k}<c^*$.  
\end{lemma}
\begin{proof}
Fix a nonnegative function $k_1\in C_p^{\alpha/2}([0,T])$ with $\overline{k_1}<c^*$.  By the definition of $c^*$ there exists a nonnegative function  $k_2\in C_p^{\alpha/2}([0,T])$ such that $\overline{k_1}<\overline{k_2}<c^*$, and problem \eqref{eqn-wave} admits a positive solution $\Psi^2$ when $k(t)=k_2(t)$. 
Denote
\begin{equation}\label{eq-l_i}
I_i(t):=\int_0^tk_i(s)ds \ \  \text{ for }i=1,2,  \qquad   I(t):=I_1(t)-I_2(t).
\end{equation}
Since $I_1(0)=I_2(0)$ and 
\begin{equation*}
 \lim_{t\to-\infty}\frac{I_1(t)}{t}= \overline{k_1}<\overline{k_2}=\lim_{t\to-\infty}\frac{I_2(t)}{t},
\end{equation*}
there exists $t_0\leq 0$ such that 
\begin{equation*}
  I(t)>0 \ \ \text{ for }t< t_0,  \qquad I(t_0)=0.
\end{equation*}
Define 
\begin{equation*}
  {\Psi}(t,x):=\Psi^2(t,x-I(t)) \  \text{ for } t\in[t_0-T,t_0], \ x\geq I(t).
\end{equation*}
Since $\Psi^2$ is $T$-periodic in $t$ and  strictly increasing in $x$, one has
\begin{equation*}
  \Psi^2(t_0-T,x-I(t_0-T))= \Psi^2(t_0,x-I(t_0-T))<\Psi^2(t_0,x-I(t_0)).
\end{equation*}
Thus $ {\Psi}$ satisfies
\begin{equation*}
   \begin{cases}
     \Psi_t-d\Psi_{xx}+k_1(t)\Psi_x=f(t,\Psi),&t\in[t_0-T,t_0],\ x> I(t),\\
     \Psi(t,I(t))=0, \ \ \Psi(t,\infty)=1, &t\in[t_0-T,t_0],\\
\Psi(t_0-T,x)<\Psi(t_0,x), &x\geq I(t_0-T).\\
   \end{cases}
\end{equation*}
Denote  
\begin{equation*}
\bar{\Psi}(t,x):=1, \qquad \underline{\Psi}(t,x):=\begin{cases}
  0, &t\in[t_0-T,t_0], \ 0\leq x\leq I(t),\\
   {\Psi}(t,x), &t\in[t_0-T,t_0],\  x\geq I(t).
\end{cases}
\end{equation*}
For each $L>0$,   $\bar{\Psi}$ and $\underline{\Psi}$ form a pair of upper and lower solutions of \eqref{eqn-finitewave} with $k(t)$ therein replaced by $k_1(t)$.  Using arguments similar to those in Lemma~\ref{lem-ex-unex}, we conclude that \eqref{eqn-finitewave} admits a positive solution $\tilde{\Psi}_L$ with $\underline{\Psi}(t,x)\leq \tilde{\Psi}_L\leq 1$. Moreover, the function
\begin{equation*}
  \Psi^1(t,x):=\lim_{L\to\infty}\tilde{\Psi}_L(t,x) \ \ \text{ for  }(t,x)\in[0,T]\times[0,\infty)
\end{equation*}
is well-defined, and $1\geq \Psi^1(t,x)\geq \underline \Psi(t,x)$ for $(t,x)\in[0,T]\times[0,\infty)$. It follows that $\Psi^1(t,\infty)=1$, and by the standard parabolic estimates, $\Psi^1$ solves \eqref{eqn-wave} with $k(t)$ therein replaced by $k_1(t)$.
Hence, $\Psi^1$ is a positive solution of \eqref{eqn-wave}  with  $k_1(t)$ in place of $k(t)$. By Lemma \ref{lem-ex-unex}, we further know that 
\begin{center}
$\Psi^1_x(t,0)>0$ and $\Psi^1(t,\infty)=1$\  for $(t,x)\in[0,T]\times(0,\infty)$.
\end{center}
The proof is now complete. 
\end{proof}

The following two lemmas play a key role in establishing the sharp condition for the existence of positive solutions to \eqref{eqn-wave}.

\begin{lemma}\label{TW-SW-0}
For any given nonnegative function $k\in C_p^{\alpha/2}([0,T])$, at least one of the following alternatives holds:
\begin{itemize}
\item[(a)] there exists a semi-wave solution, i.e., \eqref{eqn-wave}  has a positive solution $\Psi^k$;

\item[(b)]  there exists a traveling wave solution, i.e., the problem 
 \begin{equation}\label{eqn-trav0}
   \begin{cases}
     \Phi_t^k-d\Phi_{xx}^k+k(t)\Phi_x^k=f(t,\Phi^k), \  \Phi^k_x\geq 0, &t\in[0,T],~x\in(-\infty,\infty),\\
\Phi^k(t,-\infty)=0, \ \Phi^k(t,\infty)=1,&t\in[0,T],\\
\Phi^k(0,x)=\Phi^k(T,x), &x\in(-\infty,\infty)
   \end{cases}
  \end{equation}
admits a positive solution $\Phi^k$.
\end{itemize}
\end{lemma}
\begin{proof}
Let $\{\delta_n\}\subset (0,1/2)$ be a decreasing sequence  satisfying $\delta_n\to 0$ as $n\to\infty$, and consider the problem
 \begin{equation}\label{semi-wave-n}
   \begin{cases}
     v_t-dv_{xx}+k(t)v_x=f(t,v),&t\in[0,T],~x\in(0,\infty),\\
v(t,0)=\delta_n,~v(t,\infty)=1,&t\in[0,T],\\
v(0,x)=v(T,x),  &x\in[0,\infty).\\
   \end{cases}
  \end{equation}
Since   $0$ and $1$ are a lower and an upper solution of \eqref{semi-wave-n}, respectively, it follows from Step 1 and Step 2 in the proof of \cite[Lemma 3.4]{MDW} that,  for each $n$,  \eqref{semi-wave-n} admits a positive solution $V_n$ which is strictly increasing in $x$. By Step 3 of Case (i) in the  proof of  \cite[Theorem 1.1]{DMW25}, we further obtain that $V_n$  is the unique positive solution of \eqref{semi-wave-n} \footnote{Note that  Lemma 3.4 of \cite{MDW} and Theorem 1.1 of \cite{DMW25} are stated under the additional assumptions that $f$ is of KPP type and that $f_u(\cdot,0)>0>f_u(\cdot,1)$, respectively, but a careful inspection of their proofs shows that the arguments there do not depend on these assumptions.}. For any given $m>n$, since $V_m$ and $1$ are, respectively, a lower solution and an upper solution of \eqref{semi-wave-n}, the uniqueness of   positive solution  to \eqref{semi-wave-n}   together with  the upper and lower solution argument shows that $V_m\leq V_n$. By the strong maximum principle, we deduce  that $V_m<V_n$, namely, the sequence $\{V_n\}$ is  decreasing in $n$.  Let $x_n\in(0,\infty)$ be the unique constant such that
 \begin{equation*}
  \max_{t\in[0,T]}V_n(t,x_n)=\frac{1}{2}.
\end{equation*} 
Clearly,  $x_n$ is increasing in $n$ and $x_n\to x^*\in(0,\infty]$ as $n\to\infty$.  Denote 
\begin{equation*}
  U_n(t,x):=V_n(t,x+x_n) \  \text{ for }t\in[0,T], \ x\in[-x_n,\infty).
\end{equation*}
We consider the cases $x_n\to x^*\in(0,\infty)$  and $x_n\to\infty$ separately.

Case (i): $x_n\to x^*\in(0,\infty)$. In this case, we show that \eqref{eqn-wave} admits a positive solution $\Psi^k$.  In view of   $0\leq U_n \leq 1$, it follows from the standard parabolic interior estimates and a diagonal process of selecting subsequences that there exists some subsequence (still denoted by $\{U_n\}$) and some  function $U_\infty$ such that 
\begin{equation*}
U_n \to U_\infty \text{   in } C_{loc}^{1,2}([0,T]\times(-x^*,\infty)) \text{ as  }n\to \infty, 
\end{equation*}
 where $U_\infty$
  satisfies
  \begin{equation*}
   \begin{cases}
     U_t-dU_{xx}+k(t)U_x=f(t,U), \  0\leq U\leq 1, &t\in[0,T],~x\in(-x^*,\infty),\\
%\Phi^*(t,\infty)=1,&t\in[0,T],\\
U(0,x)=U(T,x), &x\in(-x^*,\infty).\\
   \end{cases}
  \end{equation*}
  Moreover, $U_\infty$ is monotone nondecreasing  in $x$ and 
  \begin{equation*}
  \max_{t\in[0,T]}U_\infty(t,0)=\frac{1}{2}.
\end{equation*}
Hence,  the functions
\begin{equation*}
 q(t):=\lim_{x\to \infty}U_\infty(t,x) \text{ and }U_\infty(t,-x^*):=\lim_{x\to(-x^*)^+}U_\infty(t,x)
  \end{equation*}
are well-defined. Letting $  q_n(t,x):= U_\infty(t,n+x)$, it is easily seen by standard parabolic regularity that 
$ q_n(t,x)$ converges to $  q(t)$ in $C^{1,2}_{loc}([0,T]\times \mathbb R)$, and $ q(t)$ satisfies
\begin{equation*}
  %\begin{cases}
  %p_t=f(t,p),\ 0\leq p(t)\leq \frac 12 \mbox{ for }  t\in[0,T]; &  p(0)=p(T),\\
  q_t=f(t, q),\  0\leq q(t)\leq 1 \mbox{ for } t\in [0, T]; \ q(0)=q(T),\   \max_{t\in [0, T]}q(t)\geq \frac 12.
  %\end{cases}
\end{equation*}
By ($\rm{\bf{f_p}}$) and (${\bf f_m}$), we see that  $q(t)\equiv 1$. On the other hand, by the standard parabolic estimates up to the boundary, for any fixed small  $r\in(0,1)$, there exists a constant $C=C(r)$, independent of $n$, such that
\begin{equation*}
  \|V_n\|_{C^{\frac{1+\alpha}{2},1+\alpha}([0,T]\times[0,r])}\leq C \ \text{ for all  }n.
\end{equation*}
It follows that 
\begin{equation*}
  |V_n(t,x)|=|V_n(t,x)-V_n(t,0)|\leq C|x| \text{ for all }(t,x)\in[0,T]\times[0,r] \text{ and  all   } n,
\end{equation*}
which implies that 
\begin{equation*}
  |U_n(t,x)|\leq C|x+x_n| \text{ for all }(t,x)\in[0,T]\times[-x_n,r-x_n] \text{ and all   } n,
\end{equation*}
Therefore, for sufficiently small $\delta\in(0,r)$,
\begin{equation*}
  |U_n(t,-x^*+\delta)|\leq C|-x^*+\delta+x_n| \text{ for all }t\in[0,T] \text{ and  all large  } n.
\end{equation*}
Letting $n\to\infty$ and then $\delta\to 0^+$, we obtain that 
\begin{equation*}
  U_\infty(t,-x^*)\equiv 0.
\end{equation*}
By the strong parabolic maximum principle, 
\begin{equation*}
  0<U_\infty(t,x)<1 \text{ for }t\in[0,T], \ x\in(-x^*,\infty).
\end{equation*}
Clearly, $\Psi^k(t,x):= U_\infty(t,x-x^*)$ is a solution of  \eqref{eqn-wave}.

Case (ii): $x_n\to\infty$. In this case, we prove that  \eqref{eqn-trav0} admits a positive solution $\Phi^k$.   Similarly to Case (i), up to a subsequence, we have that  
\begin{equation*}
U_n \to \tilde{U}_\infty \text{   in } C_{loc}^{1,2}([0,T]\times\mathbb R) \text{ as  }n\to \infty, 
\end{equation*}
 where $ \tilde{U}_\infty$ is monotone nondecreasing  in $x$ and  satisfies
  \begin{equation*}
   \begin{cases}
     U_t-dU_{xx}+k(t)U_x=f(t,U), \  0\leq U\leq 1, &t\in[0,T],~x\in\mathbb R,\\
U(0,x)=U(T,x), &x\in\mathbb R,\\
  \max_{t\in[0,T]}U(t,0)=\frac{1}{2}.
   \end{cases}
  \end{equation*}
Hence, the functions
\begin{equation*}
 p(t):=\lim_{x\to- \infty} \tilde{U}_\infty(t,x)  \text{ and } q(t):=\lim_{x\to \infty} \tilde{U}_\infty(t,x) 
  \end{equation*}
are well-defined and satisfy
\begin{equation*}\begin{cases}
  p_t=f(t,p),\ 0\leq p(t)\leq \frac 12 \mbox{ for }  t\in[0,T]; &  p(0)=p(T),\\
  q_t=f(t, q),\  0\leq q(t)\leq 1 \mbox{ for } t\in [0, T]; & q(0)=q(T),\   \max_{t\in [0, T]}q(t)\geq \frac 12.
  \end{cases}
\end{equation*}
It follows from  ($\rm{\bf{f_p}}$) and (${\bf f_m}$) that  $p(t)\equiv 0$ and $q(t)\equiv 1$. By the strong parabolic maximum principle, we further have
\begin{equation*}
  0<\tilde{U}_\infty(t,x)<1 \text{ for }t\in[0,T], \ x\in\mathbb R.
\end{equation*}
 Hence $\tilde{U}_\infty$ solves \eqref{eqn-trav0}.
The proof of the lemma is now complete.
\end{proof}

The function $\Phi^k(t,x)$ in   part (ii) of Lemma \ref{TW-SW-0} is known as a traveling wave for \eqref{Cauchy} with $f(t,u)$ in place of $f(u)$, and $k^*(t)$ is known as the speed function of this traveling wave. Our next result shows that a semi-wave   and a traveling wave   cannot exist simultaneously for the same nonnegative speed function $k\in C_p^{\alpha/2}([0,T])$.
 
  \begin{lemma}\label{TW-SW}
 For any given nonnegative function $k\in C_p^{\alpha/2}([0,T])$,  at most one of the alternatives in Lemma \ref{TW-SW-0} can occur.
\end{lemma}
\begin{proof}
Suppose, by contradiction, that (a) and (b) occur simultaneously for some nonnegative function $k\in C_p^{\alpha/2}([0,T])$. For
\begin{equation*}
  \lambda:=-\frac{1}{T}\int_0^Tf_u(s,1)ds>0,
\end{equation*}
by \eqref{smooth},  there exists $\delta>0$ such that 
\begin{equation}\label{eq-fu0}
      |f_u(t,u)-f_u(t,1)|\leq \frac{\lambda}{2}\quad \text{ for } t\in[0,T],~1-\delta\leq u\leq 1+\delta.
\end{equation}
 Define  
 \begin{equation}\label{eq-rho0}
 \varrho(t) :=e^{\frac{\lambda }{2}t+\int_{0}^tf_u(s,1)ds} \text{ for }t\in[0,T],\qquad \varsigma_0:=\frac{2\delta}{\|\varrho\|_{L^\infty([0,T])}}.
 \end{equation}
 Clearly, $\varrho(0)=1$ and $\varsigma_0\leq 2\delta$.
 Since $\Psi(\cdot,+\infty)=\Phi(\cdot,+\infty)=1$, there exists $\xi_0>0$ large enough such that 
\begin{equation}\label{eq-phi-0}
\max\{|\Phi(t,x)-1|,|\Psi(t,x)-1|\}<\varsigma_0/2 \quad \text{ for }t\in\mathbb R,~ x\geq\xi_0.
\end{equation} 
This together with  $\Phi(t,x),\Psi(t,x)\in (0,1)$ for $x>0$ implies that, for  $z_0>0$   large  and all $z\geq z_0$,
\begin{equation}\label{eq-comp0}
  \Psi(t,x)\leq \begin{cases}
    \Phi(t,x+z), &t\in[0,T], ~x\in[0,\xi_0],\\
    \Phi(t,x+z)+\varsigma_0/2, &t\in[0,T],~ x\in(\xi_0,\infty).
  \end{cases}
\end{equation}

Define
\begin{equation}\label{sigmin0}
  \varsigma_{min}:=\inf\{\varsigma>0:  \Psi(0,x)\leq \Phi(0,x+z)+\varsigma \text{ for all }x\geq 0, z\geq z_0\}.
\end{equation}
We show that  
\begin{equation}\label{eq-claim}
\varsigma_{min}=0.
\end{equation}
By \eqref{eq-comp0} and the fact that $\Psi(\cdot,\infty)=\Phi(\cdot,\infty)$=1, one has   $ \varsigma_{min}\in[0,\varsigma_0/2]$. For fixed $z\geq z_0$, define
 \[\Phi^{z}(t,x):=\Phi(t,x+z)+\varsigma_{min} \varrho(t).\]
Using \eqref{eq-rho0} and  \eqref{eq-fu0}, we have
\begin{align*}
  \mathcal{N}^1[\Phi^{z}]:&=\Phi^{z}_t-d\Phi^{z}_{xx}+k^*(t)\Phi^{z}_x-f(t,\Phi^{z})\\
  &=\varsigma_{min}  \varrho'(t)+f(t,\Phi)-f(t,\Phi^{z})\nonumber \\
&= \Big[f_u(t,1)+\frac{\lambda}{2}-\int_0^1f_u(t,\Phi+\theta \varsigma_{min}  \varrho (t)  )d\theta\Big]\varsigma_{min}  \varrho(t)\\
&\geq 0 \quad \text{ for }t\in[0,T], x\in[\xi_0,\infty).
\end{align*}
This and \eqref{eq-comp0}, \eqref{sigmin0} allow us to apply the comparison principle to $\Psi(t,x)$ and $\Phi^{z}(t,x)$ over the region $[0,T]\times(\xi_0,\infty)$ to  obtain 
\begin{equation*}
  \Psi(t,x)\leq \Phi^{z}(t,x)\text{ for }t\in(0,T], x\in (\xi_0,\infty).
\end{equation*}
In particular, 
\begin{equation}\label{eq-psi120}
  \Psi(T,x)\leq \Phi^{z}(T,x)=\Phi(T,x+z)+\varsigma_{min} \varrho(T) \text{ for } x\in (\xi_0,\infty), z\geq z_0.
\end{equation}
Since  $\Psi$ and $\Phi$ are periodic in $t$, it follows that
\begin{equation*}
  \Psi(0,x)\leq  \Phi(0,x+z)+\varsigma_{min} \varrho(T) \text{ for } x\in (\xi_0,\infty),  z\geq z_0.
\end{equation*}
This, together with \eqref{eq-comp0} and the definition of  $\varsigma_{min}$, implies that $\varsigma_{min}\leq \varsigma_{min}  \varrho(T)$. Due to $ \varrho(T)<1$, we must have $\varsigma_{\min}=0$, as desired. 

Now it follows from \eqref{eq-claim} that
\begin{equation*}
  \Psi(0,x)\leq \Phi(0,x+z) \text{ for } x\geq 0, z\geq z_0.
\end{equation*}
Define
\begin{equation*}
  z_{min}:=\inf\{\tilde{z}\in\mathbb R:\Psi(0,x)\leq \Phi(0,x+z) \text{ for } x\geq 0, z\geq \tilde{z}\}.
\end{equation*}
Since $\Phi(\cdot,-\infty)=0<\Psi(\cdot,x)$ for $x>0$, it is obvious that $z_{min}\in (-\infty,z_0]$ is well-defined.  Set   $W^z(t,x):=\Phi(t,x+z)-\Psi(t,x)$ for $z\geq z_{min}$. Then $W=W^z$  satisfies
\begin{equation*}
  \begin{cases}
    W_t-dW_{xx}+k^*(t)W_x-c(t,x)W=0, &t\in[0,T], x\in(0,\infty),\\
    W(t,0)>0, &t\in[0,T],\\
    W(0,x)\geq 0, &x\in[0,\infty),
  \end{cases}
\end{equation*}
where
\begin{equation*}
  c(t,x):=\int_0^1f_u(t,s\Psi(t,x) +(1-s)\Phi(t,x+z))ds
\end{equation*}
is bounded.
By the strong parabolic maximum principle, we have 
\begin{equation*}
  W^z(t,x)>0 \text{ for }t\in(0,T], x\in[0,\infty), z\geq z_{min},
\end{equation*}
This, combined with the periodicity of $\Psi$ and $\Phi$, yields that 
\begin{equation*}
  \Psi(t,x)< \Phi(t,x+z)\ \text{ for }t\in[0,T], x\in[0,\infty), z\geq z_{min}.
\end{equation*} 
Hence, there exists $\epsilon>0$ small enough such that  \eqref{eq-comp0} holds  for all  $z\geq z_{min}-\epsilon$. Arguing as in the proof of \eqref{eq-claim},  we can further show that 
\begin{equation*}
  \Psi(0,x)\leq \Phi(0,x+z)~~ \text{ for } x\geq 0, z\geq z_{min}-\epsilon.
\end{equation*}
But this contradicts the definition of $z_{min}$. Hence, (a) and (b) cannot occur simultaneously, and the proof is complete.
\end{proof}

We are now in a position to use $c^*$ to  determine exactly when  \eqref{eqn-wave} has a positive solution.

\begin{lemma}\label{lem-iff}
Given any nonnegative function $k\in C_p^{\alpha/2}([0,T])$, problem \eqref{eqn-wave} admits a  positive solution $\Psi^k$ if and only if $\bar{k}<c^*$.
\end{lemma}

\begin{proof}
  By Lemma \ref{lem-suff} and the definition of $c^*$, it suffices  to show that for any given nonnegative function $k\in C_p^{\alpha/2}([0,T])$, problem \eqref{eqn-wave} admits no bounded positive solution whenever $\bar{k}=c^*$. 
  
 Suppose that $k_*\in C_p^{\alpha/2}([0,T])$ is a nonnegative function  such that $\overline{k_*}=c^*$. Let $\{\epsilon_n\}\subset  (0,1)$ and $\{\delta_n\}\subset (0,1/2)$ be two sequences such that $\epsilon_n\searrow  0$ and $\delta_n\searrow 0$ as $n\to\infty$, and  let $V^m_n$ be the unique solution of \eqref{semi-wave-n} with $k(t)=k_m(t):=k_*(t)+\epsilon_m$. By the proof of Lemma \ref{TW-SW-0}, for each $n,m\geq 1$, $V^m_n$ is strictly increasing in $x$ and there exists a unique $x^m_n\in(0,\infty)$ such that 
\begin{equation*}
 \max_{t\in[0,T]} V^m_n(t,x^m_n)=\frac{1}{2}.
\end{equation*}
 Since $\overline{k_m}>c^*$, \eqref{eqn-wave} admits no positive solution with $k=k_m$. It follows from the proof of Lemma \ref{TW-SW-0}  that   $x^m_n\to\infty$ as $n\to\infty$ for each $m$.
Therefore, we can find a subsequence $\{n_m\}$  of positive integers  such that $n_m\geq m$ and  $x^m_{n_m}>m$. So $  x^m_{n_m}\to\infty$ as $m\to\infty$. 
Since  $0\leq V^m_{n_m}\leq 1$ and $k_*\leq k_m\leq k_*+1$,  an argument similar to that in the proof of Lemma \ref{TW-SW-0} shows that, up to a subsequence, 
\begin{center}
$V^m_{n_m}(t,x+x^m_{n_m})\to V_\infty(t,x)$ as $m\to\infty$ in $C_{loc}^2([0,T]\times\mathbb R)$,
\end{center} 
where $V_\infty$ is nondecreasing in $x$  and  
$$ \max_{t\in[0,T]}V_\infty(t,0)=\frac{1}{2}.$$
Thus, we can argue as in the proof of Lemma \ref{TW-SW-0} to conclude that $V_\infty$ is a positive  solution of \eqref{eqn-trav0} with   $k=k_*$. By Lemma \ref{TW-SW}, \eqref{eqn-wave} admits no positive solution with $k=k_*$.
\end{proof}

\begin{proof}[{\bf Proof of Theorem \ref{theo-semi-wave}}]
By Lemma \ref{lem-suff}, for any given nonnegative function $k\in C_p^{\alpha/2}([0,T])$ with $\bar{k}<c^*$, problem \eqref{eqn-wave} admits a positive solution $\Psi^k(t,x)$, which satisfies \eqref{prop}. As in Step c of  Case (i) in Step 2 below, one concludes that $\Psi^k$ is the unique positive solution of \eqref{eqn-wave} for each given such $k$. By an argument analogous to that in the proof of \cite[Proposition~2.1]{DGP},  we see  that $\Psi^{k}$ is  decreasing in $k(t)$ in the following sense: 
\begin{eqnarray}\label{eq-nondec}
\ \ \ k_1\leq, \not\equiv k_2\implies   \Psi^{k_1}(t,x)>  \Psi^{k_2}(t,x),\   \Psi_x^{k_1}(t,0)>\Psi_x^{k_2}(t,0)\ \text{ for }t\in[0,T],\ x\in(0,\infty).
\end{eqnarray}

For clarity, we divide the rest of the proof into two steps.

{\bf Step 1:} \underline{Existence}.
We  show that for every $\mu>0$, there exists a positive function $k\in C_p^{\alpha/2}([0,T])$  with $0<\bar{k}< c^*$ such that  \eqref{eqn-wave} has a positive solution $\Psi:=\Psi^k$ and \eqref{eq-phix1} holds. 

For any given nonnegative $k\in C_p^{\alpha/2}([0,T])$, let $\hat{\Psi}^k$ be the unique maximal nonnegative $T$-periodic solution of \eqref{eqn-wave}.   By the strong maximum principle, uniqueness of positive solution to \eqref{eqn-wave} and Lemma \ref{lem-ex-unex}, we see that $\hat{\Psi}^k \equiv \Psi^k$ whenever \eqref{eqn-wave} has a nontrivial nonnegative solution $\Psi^k$; otherwise, $\hat{\Psi}^k \equiv 0$. Let
\begin{equation*}
  H:=\{k\in C^{\alpha/2}_p([0,T]):k\geq0 \}
\end{equation*}
and  $\mathcal{F}$ be the operator given by
\begin{equation*}
  \mathcal{F}[k]=\mu \hat{\Psi}_x^k(\cdot,0) \ \text{ for }k\in H.
\end{equation*}

Taking $k(t)\equiv0$, we easily see that $\Lambda^{0}_{0}=-\bar{a}_0<0$, and
it follows from Lemma \ref{lem-ex-unex} that $\hat{\Psi}^0>0$ and 
\begin{equation*}
  \mathcal{F}[0](t)=\mu \hat{\Psi}_x^0(t,0)>0 \text{ for }t\in[0,T].
\end{equation*}
Define
\begin{equation*}
k_0(t):= \mathcal{F}[0](t)  \ \text{ and } \ H_0:=\{k\in H:  k(t)\leq k_0(t)\text{ for }t\in[0,T]\}.
\end{equation*}
Then $H_0$ is closed and convex. Since $\hat{\Psi}_x^k(t, 0)$ is decreasing in $k$ by \eqref{eq-nondec}, the operator $\mathcal{F}$ maps $H_0$ to itself. In what follows, we  apply the Schauder fixed point theorem to find a  fixed point $k(t)$ of $\mathcal{F}$ in $H_0$, which clearly gives a solution pair $(k(t), \hat{\Psi}^k(t, x))$ to \eqref{eqn-wave} and \eqref{eq-phix1}. To achieve this, it suffices to show that $\mathcal{F}$ is continuous and $\mathcal{F}[H_0]$ is precompact. 

To prove the continuity of $\mathcal{F}$, let $\{k_n\}\subset H_0 $ be any sequence such that $k_n\to k$ in $H_0$, and we want to show that 
\begin{equation}\label{eq-con}
  \mathcal{F}[k_n]\to\mathcal{F}[k] \text{ in }C^{\alpha/2}([0,T]) \text{ as }n\to\infty.
\end{equation}
As $0\leq \hat{\Psi}^{k_n}\leq 1$, it follows from the standard parabolic estimates and a diagonal process of selecting subsequences that, up to extraction of a subsequence, 
\begin{equation*}
  \hat{\Psi}^{k_n}\to U \text{   in } C_{loc}^{1,2}([0,T]\times[0,\infty)) \text{ as }n\to\infty,
\end{equation*}
where $U$ is a nonnegative solution of \eqref{eqn-wave}. By the definition of $\mathcal{F}$, it suffices to prove that $U\equiv \hat{\Psi}^k$.  If $U\not\equiv0$, the strong parabolic maximum principle yields that $U(\cdot,x)>0$ for $x>0$.  By the uniqueness of the positive solution to \eqref{eqn-wave}, it is clear that  $U\equiv \hat{\Psi}^k$. If $U\equiv0$, we need to show that \eqref{eqn-wave} admits no positive solution. Suppose by contradiction that \eqref{eqn-wave} has a positive solution $\hat{\Psi}^k$; then Lemma \ref{lem-iff} implies that $\bar{k}<c^*$. Thus, there exists a constant $\epsilon^*>0$ small enough   such that $\overline{k+\epsilon^*}<c^*$. By Lemma \ref{lem-iff},  problem  \eqref{eqn-wave} with $k$ replaced by $k+\epsilon^*$ has a unique positive solution $\hat{\Psi}^{k+\epsilon^*}$. Since $k_n\to k$ as $n\to\infty$, we see that $k_n<k+\epsilon^*$ for all large $n$.   Then Lemma \ref{lem-iff} shows that $\hat{\Psi}^{k_n}$ is the unique positive solution of \eqref{eqn-wave} with $k=k_n$ for all large $n$. It follows from \eqref{eq-nondec} that $\hat{\Psi}^{k_n}\geq \hat{\Psi}^{k+\epsilon^*} $ for all large $n$. This contradicts the fact that $\hat{\Psi}^{k_n}\to 0$ in  $C_{loc}^{1,2}([0,T]\times[0,\infty))$ as $n\to\infty$. Therefore,  \eqref{eq-con} holds and  $\mathcal{F}$ is continuous on $H_0$.

We next show that $\mathcal F[H_0]$ is precompact on $H_0$. Let $\{k_n\}$ be any bounded sequence in $H_0$.  
 We denote $\hat{\Psi}^{k_n}(t,x)$ by $\hat{\Psi}^{n}(t,x)$ for simplicity. Since
\begin{equation*}
  0\leq k_n(t)\leq k_0(t), \  0\leq \hat{\Psi}^n(t,x)\leq 1 \text{ for }t\in[0,T], \ x\in[0,\infty),
\end{equation*}
we see that $k_n$ and $\hat{\Psi}^n$ are uniformly bounded in $L^\infty$ with respect to $n$. By the standard parabolic $L^p$ estimates, for any $p>1$ and compact set $D\subset [0,T]\times[0,\infty)$, $\|\hat{\Psi}^n\|_{W_p^{1,2}(D)}$ is  uniformly bounded with respect to $n$.  It then follows from the Sobolev embedding theorem and a diagonal process of selecting subsequences that, up to extraction of a subsequence, \begin{equation*}
  \hat{\Psi}^n\to \tilde{\Psi} \text{   in } C_{loc}^{(1+\alpha)/2,1+\alpha}([0,T]\times[0,\infty)) \text{ as }n\to\infty \mbox{ for some } \alpha\in (0,1).
\end{equation*}
Hence, 
\[\mathcal{F}[k_n]=\mu \hat{\Psi}_x^n(\cdot,0)\to\mu\tilde{\Psi}_x(\cdot,0) \ \text{ in } C^{ \alpha/2}([0,T]).\]
This implies that $\mathcal{F}[H_0]$ is precompact. Therefore, by the Schauder fixed point theorem, there exists $k^*\in H_0$ such that $\mathcal{F}[k^*]=k^*$. Clearly,  $k^*\not\equiv 0$ since $\mathcal{F}[0]=k_0>0$. Hence, $\hat{\Psi}_x^*(\cdot,0)\geq,\not\equiv 0$ and $\hat{\Psi}^{k^*}$ is a nontrivial nonnegative solution. From the strong parabolic maximum principle and Hopf boundary lemma, we see that 
\begin{equation*}
  k^*(t)=\mu \hat{\Psi}_x^{k^*}(t,0)>0 \text{ for }t\in[0,T],
\end{equation*}
and $\hat{\Psi}^{k^*}= \Psi^{k^*}$ is the positive solution of \eqref{eqn-wave} with $k=k^*$. This completes Step 1.

{\bf Step 2:} \underline{Uniqueness}. We show that the pair $(k^*,\Psi^{k^*})$ obtained in Step 1 is  unique.
This can be proved by adapting the arguments in the proof of \cite[Theorem 1.1]{DMW25}.  Since the required changes are rather obvious, we only sketch the main steps.

 Suppose that both $(k_1,\Psi^1)$ and $(k_2,\Psi^2)$ solve  \eqref{eqn-wave} and \eqref{eq-phix1},  and so we have
 \begin{equation}\label{eq-phi12}
    \mu\Psi^i_x(t,0)= k_i(t)\quad  \text{ for }t\in[0,T]\text{ and } i=1,2.
 \end{equation}
Denote 
\begin{equation*}
  \lambda:=-\frac{1}{T}\int_0^Tf_u(s,1)ds>0.
\end{equation*}
Owing to the regularity of $f$ in $u$,  there exists $\delta>0$ such that 
\begin{equation}\label{eq-fu}
      |f_u(t,u)-f_u(t,1)|\leq \frac{\lambda}{2}\quad \text{ for } t\in[0,T],~1-\delta\leq u\leq 1+\delta.
\end{equation}
Set
\begin{equation*}
  \overline{k_i}:=\frac{1}{T}\int_{0}^{T}k_i(t)dt, ~~I_i(t):=\int_0^tk_i(s)ds \quad \text{ for }i=1,2
\end{equation*}
and 
\begin{equation*}
  I(t):=I_2(t)-I_1(t).
\end{equation*}
%For any $\sigma\in(0,\sigma_0]$, the function $\Phi^{i,\sigma}(t,x):=\Phi^i(t,x)+\sigma \rho(t), i=1,2,$ satisfies
%\begin{align}\label{eq-uppso}
 % \mathcal{N}^i[\Phi^{i,\sigma}]&:=\Phi^{i,\sigma}_t-d\Psi^{i,\sigma}_{xx}-k_i(t)\Psi^{i,\sigma}_x-f(t,\Psi^{i,\sigma})=\sigma \rho_t+f(t,\Psi^i)-f(t,\Psi^{i,\sigma})\nonumber \\
%&\geq [f_u(t,1)+\frac{\lambda}{2}-\int_0^1f_u(t,\Psi^{i}+\theta \sigma \rho(t))d\theta]\sigma q(t)\geq 0 ~~\text{ for }t\in[0,T], x\in(\xi_0,\infty).
%\end{align}
% Hence,  $\Psi^{i,\sigma}, i=1,2$, is an upper solution of $  \mathcal{N}^i[u]=0$ in $[0,T]\times(\xi_0,\infty)$.
We are going to show that $k_1\equiv k_2$ and $\Psi_1\equiv\Psi_2$ by dividing the arguments into two cases.

\textbf{Case (i):} $\overline{k_1} =\overline{k_2}$. In this case, we deduce that $k_1\equiv k_2$ and $\Psi^1\equiv\Psi^2$.

Due to  $\overline{k_1} =\overline{k_2}$,  we see that $I(t)$ is $T$-periodic in $t$ and $I(0)=I(T)=0$, and hence there exists $t_0\in[0,T]$ such that 
\begin{equation*}
M_0:=\min_{t\in \mathbb R}I(t)=I(t_0+nT)\leq 0 \text{ for }n\in\mathbb Z.
\end{equation*}
 It follows that $ I(t)-M_0\geq 0$ for  $t\in\mathbb R$ and  $I(t_0+nT)-M_0=0$ for $n\in\mathbb Z$. Thus, 
 \begin{equation}\label{eq-k12}
 I'(t_0+nT)=k_1(t_0+nT)-k_2(t_0+nT)=0 \text{ for }n\in\mathbb Z.
 \end{equation}
 Define  
 \begin{equation}\label{eq-rho}
 \varrho(t) :=e^{\frac{\lambda }{2}(t-t_0+T)+\int_{t_0-T}^tf_u(s,1)ds} \text{ for }t\in[t_0-T,t_0],\qquad \varsigma_0:=\frac{2\delta}{\|\varrho\|_{L^\infty([t_0-T,t_0])}}.
 \end{equation}
 Clearly, $\varrho(t_0-T)=1$ and $\varsigma_0\leq 2\delta$.
 Since $\Psi^i(\cdot,+\infty)=1$, there exists $\xi_0>0$ large enough such that 
\begin{equation}\label{eq-phi-1}
|\Psi^i(t,x)-1|<\varsigma_0/2 \quad \text{ for }t\in\mathbb R,~ x\geq\xi_0.
\end{equation} 
Moreover,  since $0<\Phi^i(\cdot,x)<1$ for $x>0$, there exists  $z_0>0$   large enough such that for all $ z\geq z_0$,
\begin{equation}\label{eq-comp}
  \Psi^1(t,x)\leq \begin{cases}
    \Psi^2(t,x+I(t)-M_0+z), &t\in[t_0-T,t_0], ~x\in[0,\xi_0],\\
    \Psi^2(t,x+I(t)-M_0+z)+\varsigma_0/2, &t\in[t_0-T,t_0],~ x\in(\xi_0,\infty).
  \end{cases}
\end{equation}
We complete the proof in three steps.

{ \em Step a.} We show that 
\begin{equation}\label{sigmin}
  \varsigma_{min}:=\inf\{\varsigma>0:  \Psi^1(t_0-T,x)\leq \Psi^2(t_0-T,x+z)+\varsigma \text{ for }x\geq 0, z\geq z_0\}=0.
\end{equation}
This is similar to the argument in \cite{DMW25} and we omit the details.

{\em Step b. } We show that 
\begin{equation*}
  z_{min}:=\inf\{\tilde{z}>0:\Psi^1(t_0-T,x)\leq \Psi^2(t_0-T,x+z) \text{ for } x\geq 0, z\geq \tilde{z}\}=0.
\end{equation*}
This is again very similar to the argument in \cite{DMW25} and so we omit the details.

\emph{Step c.} We show that $k_1\equiv k_2$ and  $\Psi^1\equiv \Psi^2$.

From Step b and the fact that  $I(t)-M_0\geq 0$, one has
\begin{equation*}\begin{cases}
  \Psi^1(t_0-T,x)\leq \Psi^2(t_0-T,x)=\Psi^2(t_0-T, x+I(t_0-T)-M_0)=\Psi^{2,0}(t_0-T,x) \text{ for } x\geq 0,\\
  \Psi^1(t,0)=0=\Psi^2(t,0)\leq \Psi^2(t,I(t)-M_0)=\Psi^{2,0}(t,0) \text{ for }t\in[t_0-T,t_0].
  \end{cases}
\end{equation*}
If
\begin{equation}\label{eq-IP}
  I(t)\equiv M_0 \text{ for }  t\in\mathbb R  \  \text{ and }\ \Psi^1(t_0-T,x)\equiv\Psi^2(t_0-T,x) \text{ for }x\geq 0,
\end{equation}
then $I(t)\equiv I(0)=0$ (which yields $k_1\equiv k_2$) and by the uniqueness of positive solution to \eqref{eqn-wave} for each $k$ we see that  $\Psi^1\equiv \Psi^2$, as desired.
Otherwise, by applying the parabolic strong maximum principle to $\Psi^1(t,x)$ and $\Psi^2(t,x+I(t)-M_0)$ over $(t_0-T,t_0]\times(0,\infty)$, we have
\begin{equation*}
  \Psi^1(t,x)<\Psi^{2,0}(t,x) \text{ for }t\in(t_0-T,t_0],x\in(0,\infty).
\end{equation*}
Due to 
\begin{equation*}
 \Psi^1(t_0,0)=0=\Psi^2(t_0,0)=\Psi^{2,0}(t_0,0),
\end{equation*}
it follows from the Hopf boundary lemma that 
\begin{equation*}
  \Psi^1_x(t_0,0)<\Psi^{2,0}_x(t_0,0)=\Psi^{2}_x(t_0,0),
\end{equation*}
which in turn implies $k_1(t_0)<k_2(t_0)$ by \eqref{eq-phi12}. This contradicts \eqref{eq-k12}.

\textbf{Case (ii):} $\overline{k_1} \neq \overline{k_2}$. In this case, we derive a contradiction. Without loss of generality, we assume that $\overline{k_1}>\overline{k_2}$. Then 
\begin{equation*}
    \lim_{t\to-\infty}\frac{I_1(t)}{t}=\overline{k_1}>\overline{k_2} =\lim_{t\to-\infty}\frac{I_2(t)}{t}.
\end{equation*}
Since  $I_1(0)=I_2(0)=0$, we can find  $t_0<0$ such that 
\begin{equation*}
I(t):=I_2(t)-I_1(t)>0 \text{ for }t<t_0,  \quad   \quad I(t_0)=0,    
\end{equation*}
which implies that 
\begin{equation}\label{eq-k122}
  I'(t_0)=k_2(t_0)-k_1(t_0)\leq 0.
\end{equation}
Similarly to Case (i),  there exists $z_0$ large enough so that for each $z\geq z_0$,
\begin{equation}\label{eq-p0x}
  \Psi^1(t,x)\leq \begin{cases}
    \Psi^2(t,x+I(t)+z), &t\in[t_0-T,t_0], x\in[0,\xi_0],\\
    \Psi^2(t,x+I(t)+z)+\varsigma_0, &t\in[t_0-T,t_0], x\in(\xi_0, +\infty),
  \end{cases}
\end{equation}
 where   $\varsigma_0$ and $\xi_0$ are constants given by \eqref{eq-rho} and \eqref{eq-phi-1}, respectively.
Define 
\begin{equation*}
  \tilde{\varsigma}_{min}:=\inf\{\varsigma>0:  \Psi^1(t_0-T,x)\leq \Psi^2(t_0-T,x+I(t_0-T)+z)+\varsigma\text{ for }x\geq 0, z\geq z_0\}.
\end{equation*}
As in 
 Case (i), we can show   $\tilde{\varsigma}_{min}=0$. Then, by the same argument as in Step b of Case (i), one has
 \begin{equation*}
  \Psi^1(t_0-T,x)\leq \Psi^2(t_0-T,x+I(t_0-T)+z)  \text{ for } x\in [0,\infty), z\geq 0.
\end{equation*}
Using the parabolic strong maximum principle and Hopf boundary lemma to $\Psi^1(t,x)$ and $\Psi^2(t,x+I(t))$ over $(t_0-T,t_0]\times[0,\infty)$, we  conclude that 
 \begin{equation*}
  \Psi^1(t,x)< \Psi^2(t,x+I(t)),  \  \Psi^1_x(t_0,0)<\Psi^2_x(t_0,0) \text{ for }t\in(t_0-T,t_0], x\in [0,\infty),
\end{equation*}
which implies by \eqref{eq-phi12} that 
\begin{equation*}
  k_1(t_0)<k_2(t_0).
\end{equation*}
This contradicts \eqref{eq-k122}. The proof is complete.
\end{proof}
By adapting the arguments in the proof of \cite[Lemma 3.1]{DMW25}, we have the following asymptotic estimates for the semi-wave. The detailed proof is left to the interested reader as the modifications are rather obvious.
\begin{lemma}\label{lem-1-phi}
Let  $(k,\Psi)$ be the unique semi-wave given by Theorem \ref{theo-semi-wave}, and
  \[\bar{k}:=\frac{1}{T}\int_0^Tk(s)ds,~~b_1:=-\frac{1}{T}\int_0^Tf_u(s,1)ds,\ \nu_0:=\frac{-\bar{k}+\sqrt{\bar{k}^2+4b_1d}}{2d}.\]
  Then for any $\nu\in (0,\nu_0)$, there exists $N_1=N_1(\nu)>0$ such that
  \begin{equation}\label{phi-infty}
    |1-\Psi(t,x)|+|\Psi_x(t,x)|+|\Psi_{xx}(t,x)|+|\Psi_t(t,x)|\leq N_1 e^{-\nu x} ~\text{  for  }t\in[0,T],~x\in[0,\infty).
  \end{equation}
\end{lemma}

\section{Precise propagation profile  of  \eqref{free-bound}}
%\section{Bound for $|g(t)+\int_0^tk^*(s)ds|$ and $|h(t)+\int_0^tk^*(s)ds|$}

This section is devoted to the proof of Theorem \ref{thm-exact}.    Following the method  in \cite{MDW} based on strategies developed and refined in \cite{fife77,DMZ2,D}, we first show the boundedness of  $|g(t)+\int_0^tk(s)ds|$ and $|h(t)-\int_0^tk(s)ds|$, we then prove the convergence along a time sequence $t_n\to\infty$, and finally we prove the convergence for $t\to\infty$.    Several nontrivial modifications are needed to address the difficulties caused by the weaker condition \eqref{eq-f01}.  We first collect some preliminary results.

\begin{lemma}\label{lem-comp}
Suppose that  $\tau\in(0,\infty)$ and   $(u,g,h)$ is a solution of \eqref{free-bound}. If $\bar{g},\bar{h}\in C^1([0,\tau])$,    $\Sigma_\tau=\{(t,x)\in\mathbb{R}^2:0<t\leq \tau,\bar{g}(t)< x<\bar{h}(t)\}$,  and $\bar{u}\in C(\bar{\Sigma}_\tau)\cap C^{1,2}( \Sigma_\tau)$ satisfy
 \begin{equation}\label{eq-com}
   \begin{cases}
     \bar{u}_t-d\bar{u}_{xx}\geq f(t,\bar{u}), &0<t\leq \tau, \ \bar{g}(t)<x<\bar{h}(t),\\
     h(t)>\bar{g}(t)\geq g(t),\   \bar{u}(t,\bar{g}(t))\geq u(t,\bar{g}(t)), &0<t\leq \tau,\\
       \bar{u}(t,\bar{h}(t))=0, \ \bar{h}'(t)\geq -\mu\bar{u}_x(t,\bar{h}(t)),&0<t\leq \tau,\\
    h_0<\bar{h}(0), \tilde u_0(x)\leq,\not\equiv  \bar{u}(0,x),&\bar{g}(0)\leq x\leq \bar h(0),
   \end{cases}
 \end{equation}
 where $\tilde u_0$ is the zero extension of $u_0$ from $[-h_0, h_0]$ to $\mathbb R$,
 then
 \begin{align*}
   h(t)<\bar{h}(t),\ u(t,x)< \bar{u}(t,x)\text{ for }t\in(0,\tau] \text{ and }\bar{g}(t)< x< h(t).
 \end{align*}

 \begin{proof}
  The proof is similar to that of \cite[Lemma 3.7]{DGP}, so we omit the details here.
\end{proof}
  
 \end{lemma}
 \begin{lemma}\label{lem-comp2}
  Suppose that $\tau\in(0,\infty)$ and $(u,g,h)$ is a solution of \eqref{free-bound}.  If $\bar{g},\bar{h}\in C^1([0,\tau])$ and  $\bar{u}\in C(\bar{\Sigma}_\tau)\cap C^{1,2}( \Sigma_\tau)$ with $\Sigma_\tau=\{(t,x)\in\mathbb{R}^2:0<t\leq \tau,\bar{g}(t)< x<\bar{h}(t)\}$ satisfy
 \begin{equation*}
   \begin{cases}
     \bar{u}_t-d\bar{u}_{xx}\geq f(t,\bar{u}), &0<t\leq \tau, \ \bar{g}(t)<x<\bar{h}(t),\\
   \bar{u}(t,\bar{h}(t))=0, \ \bar{h}'(t)\geq -\mu \bar{u}_x(t,\bar{h}(t)),&0<t\leq \tau,\\
   \bar{u}(t,\bar{g}(t))=0, \ \bar{g}'(t)\leq -\mu \bar{u}_x(t,\bar{g}(t)),&0<t\leq \tau,\\
    [-h_0,h_0]\subset (\bar{g}(0),\bar{h}(0)),\  \tilde u_0(x)\leq,\not\equiv  \bar{u}(0,x),&x\in[\bar g(0), \bar h(0)],
   \end{cases}
 \end{equation*}
 where $\tilde u_0$ is the zero extension of $u_0$ from $[-h_0, h_0]$ to $\mathbb R$, then
 \begin{align*}
  [g(t),h(t)]\subset (\bar{g}(t),\bar{h}(t)),  \ u(t,x)< \bar{u}(t,x)\text{ for }t\in(0,\tau] \text{ and }g(t)<x<h(t).
 \end{align*}
 \end{lemma}

\begin{proof}
  The proof is similar to that of \cite[Lemma 4.1]{DGP} and is therefore omitted.
\end{proof}

\begin{lemma}\label{lem-uhbd}
 Suppose that $(u,g,h)$ is a solution of \eqref{free-bound}. Then there exists $K_0>0$ such that 
 \begin{equation*}
  u(t,x), \  -g'(t),\  h'(t)\in (0, K_0)\  \text{ for }t\geq 0,\  g(t)<  x<h(t).
 \end{equation*} 
 Moreover, 
 \begin{equation*}
  \limsup_{t\to\infty}u(t,x)\leq 1 \text{ uniformly in } x\in\mathbb R.
 \end{equation*}
\end{lemma}
 \begin{proof}
The proof is similar to that in \cite{DGP}, based on the construction of appropriate upper and lower solutions. The details is left to the interested reader.
 \end{proof}
 
 Let $(k,\Psi)$ be the semi-wave given by Theorem {\rm\ref{theo-semi-wave}} and $(u,g,h)$ be  a spreading  solution of \eqref{free-bound} with $u_0\in I(h_0) $.
 In what follows we will only prove the convergence of $h(t)$ and $u(t,x)$ for $x\in [0, h(t)]$; the convergence of $g(t)$ and $u(t,x)$ for $x\in [g(t), 0]$ follows as a consequence  since 
 \begin{equation}\label{u0(-x)}
 (\tilde u(t,x), \tilde g(t), \tilde h(t)):=(u(t,-x), -h(t), -g(t))
 \end{equation}
  is a spreading solution of \eqref{free-bound} with $\tilde u(0,x)=u_0(-x)$.

\subsection{Upper bound}
\iffalse
Setting    $v:=\Psi_x$, then $v$ satisfies the following problem in the weak sense:
\begin{equation*}
  \begin{cases}
v_t-dv_{xx}+k(t)v_r-f_u(t,\Psi(t,x))v=0, ~~ &t\in \mathbb R, x\in [0,\infty),\\
\mu v(t,0)= k(t), ~~& t\in\mathbb R. 
  \end{cases}
\end{equation*}
By \eqref{eq-ff} and the fact that $\Psi\in  C^{1+\alpha/2,2+\alpha}([0,T]\times[0,\infty))$, it follows that   $f_u(t,\Psi(t,x))\in C_{loc}^{\alpha/2,\alpha}([0,T]\times[0,\infty))$. Since $k\in C^{\alpha/2}([0,T])$, we can apply the standard  parabolic estimates  to imply that $v\in C_{loc}^{1+\alpha/2,2+\alpha}([0,T]\times[0,\infty))$. Hence,
\[k(t)=\mu v(t,0)\in C^{1+\alpha/2}([0,T]).\]
\fi
In this subsection, we derive upper bounds for $h(t)$ and $u(t,x)$. 

Let $\tilde{\Psi}\in C^{1+\frac\alpha 2,2+\alpha}(\mathbb R \times[-1,\infty))$  be a  smooth extension of $\Psi$ to  $\mathbb R\times[-1,\infty)$ such that 
\begin{equation*}
  \begin{cases}
  \tilde{\Psi}(t,x)= \Psi(t,x) & \text{ for } (t,x)\in\mathbb R\times[0,\infty),\\
    \tilde{\Psi}(t,\cdot)=\tilde{\Psi}(t+T,\cdot) & \text{ for }t\in[0,T].
  \end{cases}
\end{equation*}
Then 
\[\tilde{\Psi}_x(t,0)= \Psi_x(t,0)=\frac{1}{\mu} k(t).\]
We also assume that $f(t, u)$ has been smoothly extended to $\mathbb R\times (-2,\infty)$ so that $f$ remains $T$-periodic in $t$.

Define
\begin{align*}
  \begin{cases}
      N:= \displaystyle \frac{1}{2d}\left((1+d)\|\tilde{\Psi}\|_{C^{1,2}([0,T]\times[-1,0])}+\max_{(t,u)\in[0,T]\times[-1,0]}|f(t,u)|\right),\\
  \varGamma  (t,x):=-Nx^2+\tilde{\Psi}(t,x).
  \end{cases}
\end{align*}
Since   
\[ \varGamma_x(t,0)=\tilde{\Psi}_x(t,0)=\frac{1}{\mu} k(t)>0=\varGamma (t,0), \quad \Gamma_x(t,x)=-2Nx+\tilde{\Psi}_x(t,x),\]
we can choose $a\in(0,1)$ small enough  such that  
\[
\mbox{$\varGamma_x(t,x)\geq \tilde{\Psi}_x(t,x)>0$ for $(t,x)\in[0,T]\times[-a,0]$ and  }
  m_a:= -\max_{t\in[0,T]}\varGamma (t,-a)\in(0,1].
\]
Set
\begin{equation*}
  \bar{\Psi}(t,x):=\begin{cases}
 \varGamma (t,x)~~ \text{ for }-a\leq x\leq 0,\\
  \Psi(t,x) ~~\text{ for }x\geq 0.
  \end{cases}
\end{equation*}
Clearly,  $\bar{\Psi}\in C^{1}(\mathbb R\times[-a,\infty))$ is strictly increasing in $x$. Denote
\begin{equation}\label{eq-k+-}
   k_{\rm min}:=\min_{t\in[0,T]}k(t),  \ k_{\rm max}:=\max_{t\in[0,T]}k(t).
\end{equation}
For $(t,x)\in[0,T]\times[-a,0]$, a direct computation yields 
\begin{equation}\label{eq-upes}
  \begin{cases}
    \frac{1}{B_0}<\tilde{\Psi}_x(t,x) \leq  \bar{\Psi}_x(t,x)\equiv\varGamma _x(t,x)=-2Nx+\tilde{\Psi}_x(t,x)\leq2N a+B^*; \\
|\bar{\Psi}_t(t,x)|\equiv|\varGamma_t(t,x)|=|\tilde{\Psi}_t(t,x)|=|\tilde{\Psi}_t(t,x)-\tilde{\Psi}_t(t,0)|\leq B^*|x|^\alpha,
  \end{cases}
\end{equation}
where
\begin{equation*}
  B_0:=\frac{1 }{\min_{(t,x)\in[0,T]\times[-a,0]}\tilde{\Psi}_x(t,x)}, \quad B^*:=\|\tilde{\Psi}\|_{C^{1+\alpha/2,2+\alpha}([0,T]\times[-a,0])}.
\end{equation*}
Denote 
\begin{equation}\label{eq-sigma0}
\sigma_0:=-\frac{\overline{f_u(t,1)}}{2}=-\frac{1}{2T}\int_0^Tf_u(t,1)dt>0.
\end{equation}
Since $f\in C^1$ in $u$ uniformly with respect to $t$, there exists $\delta\in(0,m_a)$ such that 
\begin{equation}\label{eq-f_usig}
  |f_u(t,u)-f_u(t,1)|\leq \sigma_0 \text{ for }t\in[0,T], u\in[1-\delta,1+\delta].
\end{equation}
Let 
\begin{equation}\label{eq-c_q}
  C_q:=\sup_{s\in[0,T]}e^{2\sigma_0s+\int_{0}^{s}f_u(t,1)dt}
\end{equation}
and $q_0$ be a small constant such that 
\begin{equation}\label{eq-q_0}
0<q_0<\min\{C_q^{-1}\delta,\delta\}.
\end{equation}
 Since $\Psi(t,\infty)=1$ uniformly in $t\in\mathbb{R}$, there exists  $M_0>0$  such that 
\begin{equation}\label{eq-1q_0}
    \Psi(t,x)\geq 1-\frac{q_0}{2}>  1-\delta ~\text{ for } t\in\mathbb{R}, ~x\geq M_0.
\end{equation}
Moreover, by Lemma \ref{lem-uhbd}, we can find $\tau_0>0$  such that
\begin{equation}\label{eq-uupp}
  u(t,x)\leq 1+\frac{q_0}{2}~ \text{ for } t\geq \tau_0,\ x\in[g(t), h(t)].
\end{equation}

Define
\begin{equation}\label{eq-ups}
  \begin{cases}
    \bar{h}(t):=\int_{\tau_0}^{t}k(s)ds+\gamma(1-e^{-\alpha\sigma_0(t-\tau_0)})+h(\tau_0)+M_0+a &\text{ for } t\geq\tau_0,\\
     q(t):=q_0e^{\xi(t) } \mbox{ with } \xi(t):={\sigma_0(t-\tau_0)+\int_{\tau_0}^tf_u(s,1)ds} &\text{ for } t\geq \tau_0,\\
    \bar{u}(t,x):=\bar{\Psi}(t, \bar{h}(t)-x+x_0(t))+q(t) &\text{ for }t\geq \tau_0, x\in[g(t),\bar{h}(t)],
  \end{cases}
\end{equation}
where   $\gamma>0$ is a  constant to be specified later and $x_0(t)$ is determined by 
\begin{equation}\label{eq-x_0}
  \bar{\Psi}(t,x_0(t))+q(t)=0.
\end{equation}

Since $\sigma_0+\xi'(t)=2\sigma_0+f_u(t,1)$ is $T$-periodic and by the definition of $\sigma_0$ we have $\overline{\sigma_0+\xi'}=0$, we see that $\sigma_0(t-\tau)+\xi(t)$ is periodic in $t$. Therefore we obtain from  \eqref{eq-c_q} and \eqref{eq-q_0} that for $t\geq \tau_0$,
\begin{equation}\label{eq-qb}
\begin{aligned}
      q(t)= &\ q_0 e^{\xi(t)}
      \leq q_0 e^{-\sigma_0(t-\tau_0)}\sup_{t\in[\tau_0,\tau_0+T]}e^{\sigma_0(t-\tau_0)+\xi(t)}\\
      = &\  q_0 e^{-\sigma_0(t-\tau_0)}\sup_{s\in[0,T]}e^{\sigma_0s+\xi(s)}\leq q_0C_q e^{-\sigma_0(t-\tau_0)}\\
      \leq &\  \delta e^{-\sigma_0(t-\tau_0)}.     
\end{aligned}
\end{equation}

By the choice of $\delta$, we see that $x_0(t)\in[-a,0]$ is well-defined for $t\geq \tau_0$. Moreover, $x_0(t)\to 0$ as $t\to\infty$ and
\begin{equation}\label{eq-x0q}
 -q(t)=\bar{\Psi}(t,x_0(t))-\bar{\Psi}(t,0)=\bar{\Psi}_x(t,\theta(t))x_0(t)
\end{equation}
for some  $\theta(t)\in[x_0(t),0]$. It  follows that  
\begin{equation}\label{eq-zdt}
  |x_0(t)|\leq B_0\delta e^{-\sigma_0(t-\tau_0)}\text{ for } t\geq \tau_0.
\end{equation}
Differentiating \eqref{eq-x_0} yields
\begin{equation*}
  \bar{\Psi}_x(t,x_0(t))x_0'(t)+\bar{\Psi}_t(t,x_0(t))=-q'(t)=-(\sigma_0+f_u(t,1))q(t).
  \end{equation*}
This, combined with \eqref{eq-upes} and  $x_0(t)\in[-a,0]$, leads to
\begin{eqnarray*}
 \frac{1}{B_0}|x_0'(t)|\leq \tilde{\Psi}_x(t,x)|x_0'(t)|\leq|\bar{\Psi}_x(t,x_0(t)) x_0'(t)|\leq B^*|x_0(t)|^\alpha+ \big(\sigma_0+\|f_u(\cdot,1)\|_{L^\infty([0,T])}\big)q(t).
\end{eqnarray*}
Therefore, by  \eqref{eq-zdt} and \eqref{eq-qb} we obtain
\begin{equation}\label{eq-zd'}\begin{aligned}
  |x_0'(t)|&\leq  B_0\left[B^*B_0^\alpha\delta^\alpha e^{-\alpha\sigma_0(t-\tau_0)}+ \big(\sigma_0+\|f_u(\cdot,1)\|_{L^\infty([0,T])}\big)\delta e^{-\sigma_0(t-\tau_0)}\right]\\
  &\leq   B_1\delta^\alpha(1+\sigma_0)e^{-\alpha\sigma_0(t-\tau_0)}
\end{aligned}
\end{equation}
with  
\begin{equation*}
  B_1:= B_0(B^*B_0^{\alpha}+ \|f_u(\cdot,1)\|_{L^\infty([0,T])}+1).
\end{equation*}

\begin{lemma}\label{lemma-upp}
There exists $\gamma>0$ such that the pair $(\bar{u},\bar{h})$ satisfies 
\begin{equation*}
  \begin{cases}
    h(t)\leq \bar{h}(t) &\text{ for }t\geq \tau_0,\\
u(t,x)\leq \bar{u}(t,x) &\text{ for }t\geq\tau_0, ~x\in[g(t),h(t)].
  \end{cases}
\end{equation*}
\end{lemma}
\begin{proof} It suffices to show that $(\bar{u}(t,x),g(t),\bar{h}(t))$ is an upper solution of \eqref{free-bound} for $t\geq \tau_0$. 
  
  Clearly,  $\bar{h}(\tau_0)>h(\tau_0)$.  By \eqref{eq-1q_0}, \eqref{eq-uupp} and the monotonicity of $\Psi$,  
\begin{align*}
  \bar{u}(\tau_0,x)&\geq \Psi(\tau_0, \bar{h}(\tau_0)-h(\tau_0))+q_0\geq \Psi(\tau_0, M_0)+q_0\geq 1+\frac{q_0}{2}\geq u(\tau_0,x) \text{ for }x\in[g(\tau_0),h(\tau_0)].
\end{align*}
It is now easily seen that $\bar u(\tau_0,x)\geq \tilde u(\tau_0,x)$ for $x\in [g(\tau_0), \bar h(\tau_0)]$, where $\tilde u(\tau_0,x)$ denotes the zero extension of $u(\tau_0, x)$ from $[g(\tau_0), h(\tau_0)]$ to $\mathbb R$.
Since $g(0)<0$ and $g'(t)<0$ (by Lemma \ref{lem-uhbd}), we have $g(t)<0$ for all $t>0$.  Direct calculations give
\begin{equation}\label{eq-sum}
  \begin{aligned}
     &\bar{u}(t,\bar{h}(t))=0, ~~\bar{u}(t,g(t))\geq 0=\bar{u}(t, \bar{h}(t))=u(t,g(t)),\\
  %&\bar{u}(t,h(t))\geq   \bar{u}(t,\bar{h}(t))=0=u(t,h(t)) ~~~\text{ if }h(t)\leq \bar{h}(t),\\
  &-\mu\bar{u}_x(t,\bar{h}(t))=\mu\bar{\Psi}_z(t,x_0(t))=\mu\varGamma_z(t,x_0(t)). 
  \end{aligned}
\end{equation}
Here, $\bar{\Psi}_z$ and $\varGamma_z$ stand for the partial derivatives of $\bar{\Psi}(t,z)$ and $\varGamma(t,z)$ with respect to the variable $z$, respectively.
By \eqref{eq-zdt}, we have
\begin{eqnarray*}
   \mu\varGamma_z(t,x_0(t))&\leq& \mu \varGamma_z(t,0)+\mu  \|\varGamma\|_{C^{1,2}([0,T]\times[-a,0])}x_0(t)\\
   &\leq & k(t)+\mu B_0\delta (5N+B^*) e^{-\sigma_0(t-\tau_0)},
\end{eqnarray*}
and hence,
 \begin{eqnarray*}
  -\mu\bar{u}_x(t,\bar{h}(t))\leq k(t)+\mu B_0\delta (5N+B^*) e^{-\sigma_0(t-\tau_0)} \leq  k(t)+\alpha\gamma \sigma_0 e^{-\alpha\sigma_0(t-\tau_0)}= \bar{h}'(t)
 \end{eqnarray*}
 provided that 
 \begin{equation}\label{eq-gq1}
  \gamma\geq \frac{\mu B_0\delta(5N+B^*)}{\alpha\sigma_0}.
 \end{equation}

 We show next that 
\begin{eqnarray*}
  \mathcal{N}[\bar{u}]:=\bar{u}_t-d\bar{u}_{xx}-f(t,\bar{u})\geq 0 \text{ for }t>\tau_0,\ x\in(g(t),\bar{h}(t)).
\end{eqnarray*}
Denote $\zeta=\zeta(t,x):=\bar{h}(t)-x+x_0(t)$. We will prove the above inequality  in the regions  $-a\leq \zeta(t,x)\leq 0$ and $\zeta(t,x)> 0$ separately. For the case $-a\leq \zeta\leq 0$, we have
\begin{eqnarray*}
  &&\bar{u}(t,x)=\varGamma(t,\zeta)+ q(t)=-N\zeta^2+\tilde{\Psi}(t,\zeta)+q_0e^{\sigma_0(t-\tau_0)+\int_{\tau_0}^tf_u(s,1)ds},\\
  &&\bar{u}_t(t,x)=\varGamma_t+\varGamma_\zeta(\bar{h}'(t)+x_0'(t))+(\sigma_0+f_u(t,1))q(t),\\
  &&\bar{u}_x(t,x)=- \varGamma_\zeta=2N\zeta-\tilde{\Psi}_\zeta,~~\bar{u}_{xx}(t,x)= \varGamma_{\zeta\zeta}=-2N+\tilde{\Psi}_{\zeta\zeta}.
\end{eqnarray*}
By \eqref{eq-upes}, \eqref{eq-zd'}, \eqref{eq-qb} and the definition of $N$, 
\begin{eqnarray*}
  \mathcal{N}[\bar{u}]&=&\varGamma_t+\varGamma_\zeta(\bar{h}'(t)+x_0'(t))+(\sigma_0+f_u(t,1))q(t)+d(2N-\tilde{\Psi}_{\zeta\zeta})-f(t,\bar{u})\\
  &\geq & \tilde{\Psi}_t +\frac{1}{B_0}[k(t)+\alpha\gamma\sigma_0e^{-\alpha\sigma_0(t-\tau_0)}]-\left(2N a+B^*\right) B_1\delta^\alpha(1+\sigma_0)e^{-\alpha\sigma_0(t-\tau_0)}\\
  &&-\delta\|f_u(\cdot,1)\|_{L^\infty([0,T])} e^{-\sigma_0(t-\tau_0)}+d(2N-\tilde{\Psi}_{\zeta\zeta})-f(t,\bar{u})\\
  &\geq & \tilde{\Psi}_t +\Bigg[\frac{1}{B_0}\alpha\gamma\sigma_0-\left(2N a+B^*\right) B_1\delta^\alpha(1+\sigma_0)-\delta^\alpha\|f_u(\cdot,1)\|_{L^\infty([0,T])} \Bigg] e^{-\alpha\sigma_0(t-\tau_0)}\\
  && +d(2N-\tilde{\Psi}_{\zeta\zeta})-f(t,\bar{u})\geq0 
\end{eqnarray*}
provided that 
\begin{equation}\label{eq-gq2}
  \gamma\geq \frac{B_0\delta^\alpha}{ \alpha\sigma_0}\left[\left(2Na+B^*\right) B_1(1+\sigma_0)+ \|f_u(\cdot,1)\|_{L^\infty([0,T])} \right].
\end{equation}

 For $\zeta>0$, we have
\begin{eqnarray*}
&& \bar{u}(t,x)=\Psi(t, \zeta)+q(t)=\Psi(t, \zeta)+q_0e^{\sigma_0(t-\tau_0)+\int_{\tau_0}^tf_u(s,1)ds},\\
 &&\bar{u}_t(t,x)=\Psi_t+\Psi_\zeta(\bar{h}'(t)+x_0'(t))+(\sigma_0+f_u(t,1))q(t),\\
 &&\bar{u}_x(t,x)=- \Psi_\zeta,  ~~ \bar{u}_{xx}=\Psi_{\zeta\zeta},
\end{eqnarray*}
and  thus,
\begin{equation}\label{Nu^2}
\begin{aligned}
 \mathcal{N}[\bar{u}]&= \Psi_t+\Psi_\zeta[k(t)+\alpha\gamma\sigma_0e^{-\alpha\sigma_0(t-\tau_0)}+x_0'(t)]+(\sigma_0+f_u(t,1))q(t)-d \Psi_{\zeta\zeta}-f(t,\bar{u})\\
&=[\alpha\gamma\sigma_0e^{-\alpha\sigma_0(t-\tau_0)}+x_0'(t)]\Psi_\zeta+(\sigma_0+f_u(t,1))q(t)+f(t,\Psi)-f(t,\bar{u}).
\end{aligned}
\end{equation}
By the mean value theorem, there exists $\theta=\theta(t,x)\in(0,1)$ such that 
\begin{equation*}
  f(t,\bar{u})=f(t,\Psi)+f_u(t,\Psi+\theta q(t))q(t).
\end{equation*}
Due to  $\Psi_\zeta(t,\zeta)>0$ for $\zeta\geq 0$, we can find  positive constants  $Q_0$ and $Q_1$ such that
 \begin{equation*}
  0<Q_0\leq \Psi_\zeta(t,\zeta)\leq Q_1 ~~\text{ for }t\in\mathbb{R},\ \zeta\in[0,M_0].
 \end{equation*}
 Hence, for $\zeta\in[0,M_0]$, by \eqref{Nu^2},  \eqref{eq-zd'} and \eqref{eq-qb} we obtain
\begin{eqnarray*}
  \mathcal{N}[\bar{u}]
  &\geq& \alpha\gamma\sigma_0Q_0e^{-\alpha\sigma_0(t-\tau_0)}-B_1\delta^\alpha Q_1(1+\sigma_0)e^{-\alpha\sigma_0(t-\tau_0)}+f_u(t,1)q(t)-f_u(t,\Psi+\theta q(t))q(t)\\
&\geq& \left[\alpha\gamma\sigma_0Q_0-B_1\delta^\alpha Q_1(1+\sigma_0)-2\delta\max_{t\in[0,T],u\in[0,1+\delta]}f_u(t,u)\right]e^{-\alpha\sigma_0(t-\tau_0)}\\
&\geq& 0\ \
\end{eqnarray*}
provided that
\begin{equation}\label{eq-gq3}
  \gamma\geq \frac{\delta^\alpha \left[B_1Q_1(1+\sigma_0)+2\max_{t\in[0,T],u\in[0,1+\delta]}f_u(t,u)\right]}{\alpha\sigma_0Q_0}.
\end{equation}

It remains to prove  $\mathcal{N}[\bar{u}]\geq 0$  for   $\zeta\geq M_0$. From  \eqref{eq-1q_0} and \eqref{eq-qb} we see that 
\begin{equation*}
 1-\delta\leq  \Psi(t,\zeta)+\theta q(t)\leq 1+\delta ~~\text{ for }\zeta\geq M_0.
\end{equation*}
Hence, by using \eqref{eq-f_usig}, we deduce from  \eqref{Nu^2} that, for $\zeta\geq M_0$, 
\begin{eqnarray*}
  \mathcal{N}[\bar{u}]&\geq  &(\alpha\gamma\sigma_0-B_1\delta^\alpha(1+\sigma_0))e^{-\alpha\sigma_0(t-\tau_0)}\Psi_\zeta+(\sigma_0+f_u(t,1))q(t)-f_u(t,\Psi+\theta q(t))q(t)\\
&\geq &(\alpha\gamma\sigma_0-B_1\delta^\alpha (1+\sigma_0))\Psi_\zeta e^{-\sigma_0(t-\tau_0)}+\sigma_0q(t)+\big[f_u(t,1)-f_u(t,\Psi+\theta q(t))\big]q(t)\\
&\geq& 0\ \ 
\end{eqnarray*}
provided that 
\begin{equation}\label{eq-gq4}
  \gamma\geq \frac{B_1\delta^\alpha (1+\sigma_0)}{\alpha\sigma_0}.
\end{equation}
Therefore, once $\gamma$ is chosen sufficiently large such that \eqref{eq-gq1}, \eqref{eq-gq2}, \eqref{eq-gq3} and \eqref{eq-gq4} are satisfied, then Lemma~\ref{lem-comp} applies and we obtain
\begin{equation*}
  h(t)\leq \bar{h}(t),\ u(t,x)\leq \bar{u}(t,x) ~~\text{ for }t\geq \tau_0,~x\in[g(t),h(t)].
\end{equation*}
This completes the proof.
\end{proof}

 \subsection{Lower bound} In this subsection, we obtain the desired  lower bounds for the pair $(h(t),u(t,x))$.
 
 Let $\sigma_0$  and $C_q$ be the constants given by \eqref{eq-sigma0} and  \eqref{eq-c_q} respectively.
%\begin{equation*}
%\tilde{C}_q:=\sup_{t\in[0,T]}e^{-2\sigma_0t-\int_0^tf_u(s,1)ds}.
%\end{equation*}
Fix $q_1\in(0, \frac 12C_q^{-1})$ and define
 \begin{equation}\label{eq-low}
  \begin{cases}
  \underline{h}(t)=\int_0^{t}k(s)ds-\beta [1-e^{-\alpha(2\sigma_0-\sigma_1)t}],  &t\geq n_0T,\\
  \tilde{q}(t):=q_1e^{\sigma_1t+\int_{0}^tf_u(s,1)ds}, & t\geq n_0T,\\
  \underline{u}(t,x)=\Psi(t,\underline{h}(t)-x+x_1(t))+\Psi(t,\underline{h}(t)+x+x_1(t))-1-\tilde{q}(t), &x\in[-\underline{h}(t),\underline{h}(t)],
 \end{cases}
 \end{equation}
where $n_0\in \mathbb{N}$,  $\beta>0$ and $\sigma_1\in(\sigma_0, 2\sigma_0)$  are   constants to be specified later and the function $x_1(t)>0$ is determined by  
\begin{equation}\label{delta}
\Psi(t,x_1(t))+\Psi(t,2\underline{h}(t)+x_1(t))-1-\tilde{q}(t)=0.
 \end{equation}
 We note that since $\tilde q(t)\in (0, \frac 12)$ (due to the choice of $q_1$) and the function $x\to \Psi(t,x)+\Psi(t,2\underline{h}(t)+x)$ is strictly increasing with range covering $[1, 2)$ for $x\in [0,\infty)$, such $x_1(t)$ is well-defined.
 
 By a computation similar to \eqref{eq-qb} and using the choice of $q_1$, we obtain  
\begin{equation}\label{eq-qb2}
  \tilde{q}(t)\leq q_1C_qe^{-(2\sigma_0-\sigma_1)t} \leq \frac{1}{2}e^{-(2\sigma_0-\sigma_1)t} \text { for }t\geq 0
\end{equation}
and 
\begin{equation}\label{eq-qb3}
    \tilde{q}(t)\geq q_1\tilde{C}_qe^{-(2\sigma_0-\sigma_1)t}=:\tilde{q}_0e^{-(2\sigma_0-\sigma_1)t} \text { for }t\geq 0,
\end{equation}
where 
\begin{equation}
  \tilde{C}_q:=\min_{t\in[0,T]}e^{2\sigma_0t+\int_0^tf_u(s,1)ds}.
\end{equation}
 Since $\Psi(\cdot,\infty)=1$ and $\underline h(\infty)=\infty$, we  see that   $x_1(t)$ is bounded and $x_1(\infty)=0$. Set 
 \[ X_b:=\max_{t\in[0,\infty)}x_1(t).\]
  Then there exists a positive constant $N_2$ such that 
\begin{equation}\label{phiz-bound}
  \max\limits_{t\in[0,\infty),x\in[0,X_b]}|\Psi_{xx}(t,x)|+\max\limits_{t\in[0,\infty)}[\Psi_{t}(t,\cdot)]_{\alpha;[0,X_b]}+\max\limits_{t\in[0,\infty),x\in[0,X_b]}|\Psi_{t}(t,x)|\leq N_2.
\end{equation}
Moreover, due to $\Psi_x>0$, we can find positive constants $N_3$ and $N_4$  such that 
\begin{equation}\label{phiz-bound+}
  0<N_3\leq \Psi_{x}(t,x)\leq N_4 ~\text{ for }t\in[0,\infty),x\in[0,X_b].
\end{equation}

By  Lemma \ref{lem-1-phi}, for any given $\nu\in(0,\nu_0)$,  there exists a constant $N_1$ such that \eqref{phi-infty} holds.
Owing to
\[2\underline{h}(t)+x_1(t)\geq  k_{min}t ~\text{ for }t\geq n_0T,\]
with $n_0$ large,
we deduce from  \eqref{delta}, \eqref{eq-qb2} and  \eqref{phi-infty} that 
\begin{eqnarray*}
  \Psi(t,x_1(t))&=&1+\tilde{q}(t)-\Psi(t,2\underline{h}(t)+x_1(t))\\
  &\leq& \frac{1}{2}e^{-(2\sigma_0-\sigma_1)t}+N_1 e^{-\nu  k_{min}t}\text{ for }t\geq n_0T.
\end{eqnarray*}
The  mean value theorem then yields 
\begin{equation*}
  \Psi_x(t,\theta(t))x_1(t)=\Psi(t,x_1(t))-\Psi(t,0)\leq  \frac{1}{2}e^{-(2\sigma_0-\sigma_1)t}+N_1 e^{-\nu  k_{min} t}
\end{equation*}
for some  $\theta(t)\in[0,x_1(t)]$. Hence,  we can use  \eqref{phiz-bound+} to obtain
\begin{equation}\label{eq-xdt}
  x_1(t)\leq C_1e^{-(2\sigma_0-\sigma_1)t}+C_1 e^{-\nu k_{min}t} ~\text{ for } t\geq n_0T, 
\end{equation}
with 
\begin{equation*}
  C_1:=\max\left\{1,\frac{1 }{2N_3},\frac{N_1}{N_3}\right\}.
\end{equation*}
Since $\Psi(t,0)\equiv0$ and $\Psi\in C^{1+\alpha/2,2+\alpha}(\mathbb R\times \mathbb R^+)$, we have $\Psi_t(t,0)=0$ and  by \eqref{phiz-bound},  
\begin{eqnarray*}
  |\Psi_t(t,x_1(t))|=  |\Psi_t(t,x_1(t))-\Psi_t(t,0)|\leq \max_{t\in[0,T]}[\Psi_{t}(t,\cdot)]_{\alpha;[0,X_b]}\ x_1^\alpha(t) \leq N_2x_1^\alpha(t).
\end{eqnarray*}
  Differentiating equation \eqref{delta}  and using inequalities \eqref{phi-infty} and \eqref{phiz-bound} we obtain
\begin{eqnarray*}
  ~~&&[\Psi_x(t,x_1(t))+\Psi_x(t,2\underline{h}(t)+x_1(t))]x_1'(t)\\
  &=&-\Psi_t(t,x_1(t))-\Psi_t(t,2\underline{h}(t)+x_1(t))-2[k(t)-\alpha(2\sigma_0-\sigma_1)\beta e^{-\alpha(2\sigma_0-\sigma_1)t}]\Psi_x(t,2\underline{h}(t)+x_1(t))\\
  &&+(\sigma_1+f_u(t,1))\tilde{q}(t)\\
  &\leq&N_2x_1^\alpha(t)+N_1[1+2\alpha(2\sigma_0-\sigma_1)\beta e^{-\alpha(2\sigma_0-\sigma_1)t}]e^{-\nu  k_{min}t}+(\sigma_1+\|f_u(\cdot,1)\|_{L^\infty([0,T])})e^{-(2\sigma_0-\sigma_1)t}\\
  &\leq & (N_2C_1+\sigma_1+\|f_u(\cdot,1)\|_{L^\infty([0,T])})e^{-\alpha(2\sigma_0-\sigma_1)t}\\
  &&+[N_2C_1+N_1+2\alpha N_1(2\sigma_0-\sigma_1)\beta e^{-\alpha(2\sigma_0-\sigma_1)t}]e^{-\nu \alpha k_{min}t}\\
  &\leq &\tilde{C_1}(1+\sigma_1)e^{-\alpha(2\sigma_0-\sigma_1)t}+\tilde{C_1}[1+(2\sigma_0-\sigma_1)\beta e^{-\alpha(2\sigma_0-\sigma_1)t}]  e^{-\nu \alpha k_{min}t},
\end{eqnarray*}
where 
\begin{equation*}
  \tilde{C_1}=N_2C_1+\|f_u(\cdot,1)\|_{L^\infty([0,T])}+2N_1+1.
\end{equation*}
Combining this with \eqref{phiz-bound+}, we deduce
\begin{equation}\label{eq-xdelta'}
 x_1'(t)\leq C_2(1+\sigma_1)e^{-\alpha(2\sigma_0-\sigma_1)t}+C_2[1+(2\sigma_0-\sigma_1)\beta  e^{-\alpha(2\sigma_0-\sigma_1)t}] e^{-\nu \alpha k_{min}t} ~\text{ for }t\geq n_0T,
\end{equation}
where  $C_2:=\frac{\tilde{C_1}}{N_3}>0$.
\begin{lemma}\label{lemma-lower}
There exist constants $\beta$, $\sigma_1$ in $\mathbb R^+$ and  $n_0, n_2$ in $\mathbb{N}$ such that 
  \begin{equation*}
    \begin{cases}
  [-\underline{h}(t),\underline{h}(t)]\subset  [g(t+n_2T),h(t+n_2T)] &\text{ for }t\geq  n_0T,\\
   \underline{u}(t,x)\leq u(t+n_2T,x) &\text{ for }t\geq n_0T, ~x\in[-\underline{h}(t), \underline{h}(t)].
    \end{cases}
  \end{equation*}
  \end{lemma}

  \begin{proof} 
    It is sufficient to verify that $(\underline{u},-\underline{h},\underline{h})$ constitutes a lower solution of \eqref{free-bound}. 
    By \eqref{phiz-bound} and \eqref{eq-xdt}, we have
\begin{eqnarray*}
  \mu \Psi_x(t,x_1(t))=   \mu \Psi_x(t,0)+  \mu\Psi_{xx}(t,\theta(t))x_1(t)\geq k(t)-\mu N_2C_1(e^{-(2\sigma_0-\sigma_1)t}+ e^{-\nu k_{min}t})
\end{eqnarray*}
for some $\theta(t)\in[0,x_1(t)]$.
Hence,
\begin{align*}
  \mp\mu \underline{u}_x(t,\pm \underline{h}(t))&=\mu \left[\Psi_x(t,x_1(t))-\Psi_x(t,2\underline{h}(t)+x_1(t))\right]\\
  &\geq k(t)-\mu N_2C_1 (e^{-(2\sigma_0-\sigma_1)t}+ e^{-\nu k_{min}t})-\mu N_1 e^{-\nu  k_{min}t}\\
  &\geq k(t)- \mu N_2C_1e^{-\alpha(2\sigma_0-\sigma_1)t}-\mu (N_2C_1+N_1) e^{-\alpha\nu  k_{min}t}.
\end{align*}
It follows that 
\begin{eqnarray*}
  \underline{h}'(t)=k(t)-\alpha(2\sigma_0-\sigma_1)\beta e^{-\alpha(2\sigma_0-\sigma_1)t}\leq   \mp\mu \underline{u}_x(t,\pm\underline{h}(t)) ~\text{ for }t\geq n_0T,
\end{eqnarray*}
provided that $2\sigma_0-\sigma_1<\nu  k_{min}$ and 
\begin{equation*}
  \beta\geq \frac{\mu(2N_2C_1+N_1)}{\alpha(2\sigma_0-\sigma_1)}.
\end{equation*} 
In view of
 \begin{equation*}
 \Psi(t,\infty)=1 \text{ and }\Psi_x(t,\infty)=0 \text{  uniformly in }t\in\mathbb{R},
 \end{equation*}
 for any given small constant $\epsilon_0>0$,  we can find a sufficiently large constant $\tilde{M}_0=\tilde{M}_0(\epsilon_0)>0$  such that 
\begin{equation}\label{phi-epsi}
  \Psi(t,x)> 1-\epsilon_0, \ \Psi_x(t,x)<\epsilon_0/2~\text{ for }t\in\mathbb{R},\ x\geq  \tilde{M}_0.
\end{equation}

In the following, we show that 
\begin{equation*}
  \mathcal{N}[\underline{u}]=\underline{u}_t-d\underline{u}_{xx}-f(t,\underline{u})\leq 0.
\end{equation*}
For convenience,  we denote $\zeta^-=\zeta^-(t,x):=\underline{h}(t)-x+x_1(t)$ and   $\zeta^+=\zeta^+(t,x):=\underline{h}(t)+x+x_1(t)$. By direct computations,  
\begin{align*}
  \underline{u}_t=& \Psi_t(t,\zeta^-)+\Psi_t(t,\zeta^+)+[\Psi_z(t,\zeta^-)+\Psi_z(t,\zeta^+)](\underline{h}'(t)+x_1'(t))-\tilde{q}'(t),\\
 \underline{u}_{x}=& \Psi_{z}(t,\zeta^+)-\Psi_{z}(t,\zeta^-),\quad \underline{u}_{xx}= \Psi_{zz}(t,\zeta^-)+\Psi_{zz}(t,\zeta^+).
\end{align*}
Here $\Psi_z$ stands for the partial derivative of $\Psi(t,z) $ with respect to $z$.
Thus we can use \eqref{eqn-wave} to obtain
\begin{equation}
\begin{aligned}\label{eq-Nu_}
  \mathcal{N}[\underline{u}]=& [\Psi_z(t,\zeta^-)+\Psi_z(t,\zeta^+)][x_1'(t)-\alpha(2\sigma_0-\sigma_1)\beta e^{-\alpha(2\sigma_0-\sigma_1)t}]-(\sigma_1+f_u(t,1))\tilde{q}(t)\\
  &+f(t,\Psi(t,\zeta^-))+f(t,\Psi(t,\zeta^+))-f\left(t,\underline{u}\right).
\end{aligned}
\end{equation}
We next show that $\mathcal{N}[\underline{u}]\leq 0$ holds in regions $\underline{h}(t)-\tilde{M}_0\leq x\leq \underline{h}(t)$, $-\underline{h}(t)\leq x\leq -\underline{h}(t)+\tilde{M}_0$,  and $ -\underline{h}(t)+\tilde{M}_0\leq x\leq \underline{h}(t)-\tilde{M}_0$  separately.

\underline{Case (i)}: $\underline{h}(t)-\tilde{M}_0\leq x\leq \underline{h}(t)$. In this case, by increasing the value of $n_0$ if necessary, we have $x_1(t)\leq \zeta^-\leq x_1(t)+\tilde{M}_0$  and
\begin{equation*}
  \zeta^+=\underline{h}(t)+x+x_1(t)\geq 2\underline{h}(t)-\tilde{M}_0+x_1(t)\geq \underline{h}(t)\geq k_{min}t \mbox{ for } t\geq n_0T.
\end{equation*}
Since  $x_1(t)>0$ is  bounded,  there exist positive constants $\tilde{N}_3$ and $\tilde{N}_4$ so that 
\begin{equation*}
 0< \tilde{N}_3\leq\Psi_z(t,\zeta^+)+\Psi_z(t,\zeta^-)\leq \tilde{N}_4.
\end{equation*}
By \eqref{smooth} and  the fact that $f(\cdot,1)=0$, 
there exists $Q>0$
 such that \begin{align*}\begin{cases}
  f(t,\Psi(t,\zeta^+))= f(t,\Psi(t,\zeta^+))-f(t,1)\leq Q|\Psi(t,\zeta^+)-1|\leq QN_1e^{-\nu\zeta^+} \leq QN_1 e^{-\nu k_{min}t},\\
  f(t,\Psi(t,\zeta^-))-f(t,\underline{u})\leq Q|1-\Psi(t,\zeta^+)+\tilde{q}(t)|\leq Q[N_1 e^{-\nu  k_{min}t}+\tilde{q}(t)].  
 \end{cases}
\end{align*}
 Hence, we can use  \eqref{eq-Nu_}, \eqref{eq-xdelta'} and \eqref{eq-qb2} to obtain
\begin{eqnarray*}
  \mathcal{N}[\underline{u}]&\leq&\tilde{N}_4C_2(1+\sigma_1)e^{-\alpha(2\sigma_0-\sigma_1)t}+\tilde{N}_4C_2 [1+(2\sigma_0-\sigma_1)\beta  e^{-\alpha(2\sigma_0-\sigma_1)t}]e^{-\nu \alpha k_{min}t}\\
  &&-\tilde{N}_3\alpha(2\sigma_0-\sigma_1)\beta e^{-\alpha(2\sigma_0-\sigma_1)t}-(\sigma_1+f_u(t,1))\tilde{q}(t) +Q(2N_1 e^{-\nu  k_{min}t}+\tilde{q}(t))\\
  &\leq &\left[\tilde{N}_4C_2(1+\sigma_1)-\tilde{N}_3\alpha(2\sigma_0-\sigma_1)\beta   -\min_{t\in[0,T]}f_u(t,1) +Q\right] e^{-\alpha(2\sigma_0-\sigma_1)t}\\
  &&+\left\{\tilde{N}_4C_2\left[1+(2\sigma_0-\sigma_1)\beta  e^{-\alpha(2\sigma_0-\sigma_1)t} \right]+2QN_1\right\} e^{-\nu \alpha k_{min}t}\\
  &=& [\tilde{N}_4C_2(1+\sigma_1)-\tilde{N}_3\alpha(2\sigma_0-\sigma_1)\beta  /2-\min_{t\in[0,T]}f_u(t,1) +Q] e^{-\alpha(2\sigma_0-\sigma_1)t}\\
  &&+(\tilde{N}_4C_2+2QN_1) e^{-\nu \alpha  k_{min}t}+(2\sigma_0-\sigma_1)\beta(\tilde{N}_4C_2 e^{-\nu\alpha  k_{min}t}-\tilde{N}_3\alpha  /2)e^{-\alpha(2\sigma_0-\sigma_1)t}\\
  &\leq & 0  \ \ \text{ for } t\geq  n_0T,
\end{eqnarray*}
  provided that $2\sigma_0-\sigma_1<\nu k_{min}$, $ n_0\in\mathbb{N}$ is large enough such that 
\begin{equation*}
  \tilde{N}_3\geq 2\tilde{N}_4C_2 e^{-\nu\alpha k_{min} n_0T}/\alpha,
\end{equation*}
and 
\begin{equation*}
  \beta\geq \frac{2[\tilde{N}_4C_2(1+\sigma_1) -\min_{t\in[0,T]}f_u(t,1)  +Q+\tilde{N}_4C_2+2QN_1]}{\tilde{N}_3\alpha(2\sigma_0-\sigma_1)}
\end{equation*} 
 
\underline{Case (ii)}: $-\underline{h}(t)\leq x\leq -\underline{h}(t)+\tilde{M}_0$. This is symmetric to Case (i), and by parallel argument we have  $\mathcal{N}[\underline{u}]\leq 0$ in this case.
 
\underline{Case (iii)}: $ -\underline{h}(t)+\tilde{M}_0\leq x\leq \underline{h}(t)-\tilde{M}_0$. In this case, we have
\begin{eqnarray*}
\zeta^\pm\geq  \tilde{M}_0\  \text{ and } \ \max\{\zeta^+,\zeta^-\}\geq \underline{h}(t)+x_1(t)\geq  k_{min}t-\beta.
\end{eqnarray*} 
It follows from  \eqref{phi-epsi} and \eqref{phi-infty} that  
\begin{equation}\label{eq-mami}
  \begin{cases}
    \max\{1-\Psi(t,\zeta^+), 1-\Psi(t,\zeta^-)\}\leq \epsilon_0,\\
  0<\Psi_z(t,\zeta^+)+\Psi_z(t,\zeta^-)<\epsilon_0,\\
    \min\{1-\Psi(t,\zeta^+),1-\Psi(t,\zeta^-)\}\leq N_1e^{-\nu (k_{min}t-\beta)}.
  \end{cases}
\end{equation}
Define
\begin{equation*}
  F(\Psi(t,\zeta^+),\Psi(t,\zeta^-)):=f(t,\Psi(t,\zeta^+))+f(t,\Psi(t,\zeta^-))-f\left(t, \underline{u}\right).
\end{equation*}
By the regularity of $f$ on $u$ and the fact that $f(t,1)\equiv 0$, there exists $Q>0$ such that
\begin{align*}
 F=F(t,x):=&\ f(t,\Psi(t,\zeta^+))+f(t,\Psi(t,\zeta^-))-f\left(t, \underline{u}\right)-f(t,1)\\
 \leq & \min\Big\{Q[1-\Psi(t,\zeta^+)]+f_u(t,\theta^-)[1-\Psi(t,\zeta^+)+\tilde{q}(t)],\\
 &\ \ \ \ \ \ \ \ \ \ \ \ \ Q[1-\Psi(t,\zeta^-)]+f_u(t,\theta^+)[1-\Psi(t,\zeta^-)+\tilde{q}(t)]\Big\}\\
 \leq &\min\big\{2Q[1-\Psi(t,\zeta^+)]+f_u(t,\theta^-)\tilde{q}(t), 2Q[1-\Psi(t,\zeta^-)]+f_u(t,\theta^+)\tilde{q}(t)\big\},
\end{align*}
where 
\[
\begin{cases}\theta^-=\theta^-(t,x)\in[\Psi(t,\zeta^+)+\Psi(t,\zeta^-)-1-\tilde{q}(t),\Psi(t,\zeta^-)],\\
\theta^+=\theta^+(t,x)\in[\Psi(t,\zeta^-)+\Psi(t,\zeta^+)-1-\tilde{q}(t),\Psi(t,\zeta^+)].
\end{cases}
\]
By  \eqref{phi-epsi} and \eqref{eq-qb2}, we can choose  $n_0$ and $\tilde{M}_0$  large enough and then $\epsilon_0>0$  small enough such that for $t\geq n_0T$,
 \[|1-\theta^\pm(t,x)|\leq \delta \text{ with  } \delta \text{ given by \eqref{eq-f_usig}}.\]
It then follows from  \eqref{eq-f_usig} and  \eqref{eq-mami}  that 
\begin{align*}
 F\leq  2QN_1e^{-\nu (k_{min}t-\beta)}+(f_u(t,1)+\sigma_0)\tilde{q}(t) \text{ for }t\geq n_0T.
\end{align*}
Combining this with \eqref{eq-Nu_}, \eqref{eq-xdelta'} and \eqref{eq-qb3}, we obtain
\begin{eqnarray*}
\mathcal{N}[\underline{u}]&\leq & [\Psi_z(t,\zeta^+)+\Psi_z(t,\zeta^-)][x_1'(t)-\alpha(2\sigma_0-\sigma_1)\beta e^{-\alpha(2\sigma_0-\sigma_1)t}]\\
&&-(\sigma_1+f_u(t,1))\tilde{q}(t)+F(t,x)\\
&\leq &[\Psi_z(t,\zeta^+)+\Psi_z(t,\zeta^-)][ C_2(1+\sigma_1)-\alpha(2\sigma_0-\sigma_1)\beta]e^{-\alpha(2\sigma_0-\sigma_1)t}\\
&& +C_2\epsilon_0[1+(2\sigma_0-\sigma_1)\beta  e^{-\alpha(2\sigma_0-\sigma_1)t}]e^{-\nu \alpha k_{min}t}-(\sigma_1+f_u(t,1))\tilde{q}(t)+F(t,x)\\
&\leq &[\Psi_z(t,\zeta^+)+\Psi_z(t,\zeta^-)][ C_2(1+\sigma_1)-\alpha(2\sigma_0-\sigma_1)\beta]e^{-\alpha(2\sigma_0-\sigma_1)t}\\
&& +C_2\epsilon_0[1+(2\sigma_0-\sigma_1)\beta  e^{-\alpha(2\sigma_0-\sigma_1)t}]e^{-\nu \alpha k_{min}t}-(\sigma_1-\sigma_0)\tilde{q}(t)+2QN_1e^{-\nu (k_{min}t-\beta)}\\
&\leq &[\Psi_z(t,\zeta^+)+\Psi_z(t,\zeta^-)][ C_2(1+\sigma_1)-\alpha(2\sigma_0-\sigma_1)\beta]e^{-\alpha(2\sigma_0-\sigma_1)t}\\
&& +[C_2\epsilon_0+C_2\epsilon_0(2\sigma_0-\sigma_1)\beta  e^{-\alpha(2\sigma_0-\sigma_1)t}+2QN_1e^{\nu\beta} ]e^{-\nu \alpha k_{min}t}-(\sigma_1-\sigma_0)\tilde{q}_0e^{- (2\sigma_0-\sigma_1)t}\\
&\leq & 0   \ \text{ for }t\geq  n_0T,
\end{eqnarray*}
provided  that    $n_0$ is large enough and
\begin{equation*}
\beta\geq \frac{C_2(1+\sigma_1)}{\alpha(2\sigma_0-\sigma_1)},\  2\sigma_0-\sigma_1<\nu\alpha k_{min}.
\end{equation*} 
We have thus proved  $\mathcal{N}[\underline{u}]\leq 0$ for $t\geq n_0T$ and $x\in[-\underline{h}(t),\underline{h}(t)]$.

Because $(g(t),h(t))\to\mathbb (-\infty,\infty)$ and $u(t,x)\to 1$ as $t\to\infty$ locally uniformly in $x\in\mathbb R$, there exists a positive integer $n_1\geq  n_0$ such that
\begin{equation*}
\begin{cases}
  [-\underline{h}( n_0T),\underline{h}( n_0T)]\subset [g(n_1T), h(n_1T)],\\
   u(n_1T,x)\geq 1-\tilde{q}(n_0T)\geq \underline{u}( n_0T,x) ~ \text{ for }x\in[-\underline{h}( n_0T),\underline{h}( n_0T)].
 % u(t,x)>0~ &\text{ for }t\geq0,\ x\in(g(t),h(t)).
\end{cases}
\end{equation*} 
Set $n_2:=n_1-n_0$, then
\begin{equation*}
u( n_0T+n_2T,x)\geq \underline{u}( n_0T,x) ~\text{ for }x\in[-\underline{h}( n_0T),\underline{h}( n_0T)].
\end{equation*}  
Therefore,  taking   $0<2\sigma_0-\sigma_1<\nu\alpha k_{min}$, $\epsilon_0>0$ sufficiently small, and $\beta$, $n_0$ sufficiently large,  we can apply the lower solution version of Lemma  \ref{lem-comp2} to conclude that
\begin{eqnarray*}
  [-\underline{h}(t),\underline{h}(t)]\subset  [g(t+n_2T),h(t+n_2T)], \ \underline{u}(t,x)\leq u(t+n_2T,x)~~  \text{ for }t\geq  n_0T,x\in [-\underline{h}(t),\underline{h}(t)].
\end{eqnarray*}
The proof is now complete. 
\end{proof}

\begin{proposition}\label{prop-bound}
  There exists some constant $L>0$ such that 
  \begin{equation}\label{h-bd}
    \left|h(t)-\int_0^tk(s)ds\right|\leq L ~~\text{ for } t\geq 0.
  \end{equation}
\end{proposition}
\begin{proof} The desired estimate in \eqref{h-bd} follows easily from 
   Lemmas \ref{lemma-upp} and  \ref{lemma-lower}, since
\[
\bar h(t)=\int_0^tk(s)ds+O(1),\ \ \  \underline h(t)=\int_0^tk(s)ds+O(1).
\]
\end{proof}

\subsection{Convergence along a time sequence}

By Proposition \ref{prop-bound}, 
\begin{equation*}
  -L\leq h(t)-\int_0^tk(s)ds\leq L\text{ for } t\geq 0.
\end{equation*}
Denote
\begin{equation*}
  \begin{cases}
    K(t):=\int_0^tk(s)ds-2L, ~G(t):=g(t)-K(t),~ H(t):=h(t)-K(t),\\
  w(t,x):=u(t,x+K(t)).
  \end{cases}
\end{equation*}
Then clearly
\begin{equation}\label{eq-c3c}
\lim_{t\to\infty} G(t)=-\infty,\ \    L\leq H(t)\leq 3L,
\end{equation}
and
\begin{equation*}
u_x=w_x, ~u_{xx}=w_{xx},~  u_t=w_t-k(t)w_x.
\end{equation*}
Thus $(w(t,x), G(t),H(t))$ satisfies
\begin{equation*}
  \begin{cases}
    w_t-dw_{xx}-k(t)w_x=f(t,w), &t>0, \  G(t)<x<H(t),\\
   w(t,H(t))=0, &t>0,\\
    H'(t)=h'(t)-K'(t)=-\mu w_x(t,H(t))-k(t),&t>0.\\
  \end{cases}
\end{equation*}

Let $\{s_n\}$ be a positive sequence such that   
\begin{equation}\label{sn}
  \lim_{n\to\infty}s_n=\infty ~ \text{ and }~\lim_{n\to\infty}H(s_n)=\liminf_{t\to\infty}H(t)\in [L, 3L].
\end{equation}
Write $s_n=l_nT+r_n$, where $l_n\in \mathbb{N}$ and $r_n\in[0,T)$. Passing to a subsequence if necessary, we may assume  $r_n\to r^*$ for some $r^*\in[0,T]$.
Define
\begin{equation*}\begin{cases}
  (k_n(t),G_n(t),H_n(t)):=(k(t+s_n-r^*),G(t+s_n-r^*),H(t+s_n-r^*)), \\
  w_n(t,x):=w(t+s_n-r^*,x).
  \end{cases}
\end{equation*}
We have the following result.
\begin{lemma}\label{lem-lim}
For any given $\alpha\in (0,1)$, upon extraction of a subsequence, still denoted by $\{(H_n, w_n)\}$, we have
  \begin{equation*}
   H_n\to I \text{ in } C^{1+\alpha/2}_{loc}(\mathbb{R}) \text{  and } w_n\to W \mbox{ in } {C_{loc}^{(1+\alpha)/2,1+\alpha}(\Sigma)} ~\text{ as } n\to \infty.
  \end{equation*}
 Moreover, $(W(t,x),I(t))$ satisfies
  \begin{equation}\label{eq-V}
    \begin{cases}
      W_t-dW_{xx}-k(t)W_x=f(t,W), ~W\geq 0,  &(t,x)\in\Sigma,\\
      W(t,I(t))=0, &t\in\mathbb{R},\\
      I'(t)=-\mu W_x(t,I(t))-k(t),~ I(t)\geq I(r^*), &t\in\mathbb R,
    \end{cases}
  \end{equation}
   where $\Sigma:=\{(t,x):t\in\mathbb{R}, -\infty<x< I(t)\}$. 
\end{lemma}

\begin{proof}
Since  $|h'(t)|\leq K_0$ (by Lemma  \ref{lem-uhbd}), we have
  \begin{equation*}
    -K_0-k(t)\leq H'(t)\leq K_0-k(t) ~~\text{ for }t\geq 0.
  \end{equation*}
Define
\begin{equation*}
 \zeta=\frac{x}{H_n(t)}, ~~\tilde{W}_n(t,\zeta)=w_n(t,x).
\end{equation*}
Then $(\tilde{W}_n(t,\zeta),H_n(t))$ solves
\begin{equation}\label{eq-tilV}
  \begin{cases}
   \partial_t\tilde{W}_n-\frac{d}{H_n^2(t)}\partial_{\zeta\zeta}\tilde{W}_n-\frac{k_n(t)+H_n'(t)\zeta}{H_n(t)}\partial_\zeta\tilde{W}_n=f(t+s_n-r^*,\tilde{W}_n),&t>-s_n+r^*,\frac{G_n(t)}{H_n(t)}\leq\zeta\leq 1,\\
   \tilde{W}_n(t,1)=0,~ H_n'(t)=-\frac{\mu}{ H_n(t)}\partial_\zeta\tilde{W}_n(t,1)-k_n(t), &t>-s_n+r^*.
\end{cases}
\end{equation}
By Lemma \ref{lem-uhbd}, we see that $\tilde{W}_n$ is uniformly bounded in $\{(t,\zeta):t\geq -s_n+r^*, \frac{G_n(t)}{ H_n (t)}\leq\zeta\leq 1 \}$. Hence, we can apply the standard  parabolic $L^p$ estimates to \eqref{eq-tilV} to deduce that, for any $\tau>0$, $m>0$, $s\in\mathbb{R}$ and $p>1$, there exists $C_p>0$  independent of $n$ and $s$ such that 
\begin{equation*}
  \|\tilde{W}_n\|_{W_p^{1,2}([s-\tau,s]\times[-m,1])}\leq C_p \text{ for all large }n.
\end{equation*}
  For  $\beta\in (\alpha,1)$ and some fixed large $p$, the  Sobolev  embedding  theorem then yields 
\begin{equation}\label{L^p-est}
  \|\tilde{W}_n\|_{C^{(1+\beta)/2,1+\beta}([s-\tau,s]\times[-L,1])}\leq \tilde{C}_p  \text{ for all large }n,
\end{equation}
where  $\tilde{C}_p$ is a constant independent of $n$ and $s$.
By \eqref{eq-tilV} and \eqref{eq-c3c}, for any $s\in\mathbb{R}$, there exists $C_1>0$  independent of $n$ and $s$ such that 
\begin{equation*}
  \| H_n \|_{C^{1+\beta/2}([s-\tau,s])}\leq C_1.
\end{equation*}
Therefore, we can use a diagonal process to select a subsequence (still denoted by itself for convenience of notation) such that 
\begin{equation*}
  \tilde{W}_n\to \bar{W} \text{ in } C^{(1+\alpha)/2, 1+\alpha}_{loc}(\mathbb R\times (-\infty,1]), ~~ H_n \to I\text{ in } C_{loc}^{1+\alpha/2}(\mathbb R),
\end{equation*}
where $(\bar{W}, I)$ solves the following  problem in the $L^p$ sense and then the classical sense (by standard parabolic regularity):
\begin{equation*}
  \begin{cases}
  \bar{W}_t-\frac{d}{I^2(t)}\bar{W}_{\zeta\zeta}-\frac{k(t)+I'(t)\zeta}{I(t)}\bar{W}_\zeta=f(t,\bar{W}), \  \bar{W}\geq 0,   &t\in\mathbb R,~\zeta\in(-\infty,1],\\
  \bar{W}(t,1)=0,~ I'(t)=-\frac{\mu}{I(t)}\bar{W}_\zeta(t,1)-k(t), &t\in\mathbb R.
   \end{cases}
   \end{equation*}
Set  $W(t,x)=\bar{W}(t,x/I(t))$, then $(W,I)$ solves \eqref{eq-V} and 
\begin{equation*}
  \lim_{n\to\infty}\|w_n-W\|_{C_{loc}^{(1+\alpha)/2,1+\alpha}(\Sigma)}= 0.
\end{equation*}
In view of 
\begin{equation*}
 I(r^*)=\lim_{n\to\infty} H(s_n)=\liminf_{t\to\infty}H(t) \text{ and } I(t)=\lim_{n\to\infty}H(t+s_n-r^*),
\end{equation*}
 we further have  $I(t)\geq I(r^*)$ for $t\in\mathbb R$. The proof is thus complete.
\end{proof}

\subsection{Determine  the limit pair $(W,I)$.}

By \eqref{eq-c3c}, we have
\[L\leq I(t)\leq 3 L\ \ \text{ for }t\in\mathbb R.\]
Moreover, for $t\geq (n_0+n_2)T$ and $x\in[-\underline{h}(t-n_2T),\underline{h}(t-n_2T)]$, it follows from Lemma \ref{lemma-lower} that 
  \begin{align*}
  u(t,x)&\geq \Psi(t,\underline{h}(t-n_2T)-x+x_1(t-n_2T))\\
  &+\Psi(t,\underline{h}(t-n_2T)+x+x_1(t-n_2T))-1-\tilde{q}(t-n_2T).
\end{align*}
For simplicity, we denote
\begin{eqnarray*}
  &\Psi_n^-(t,x):=\Psi(t+r_n-r^*,\underline{h}(t+s_n-r^*-n_2T)-x-K(t+s_n-r^*)+x_1(t+s_n-r^*-n_2T)),\\
  &\Psi_n^+(t,x):=\Psi(t+r_n-r^*,\underline{h}(t+s_n-r^*-n_2T)+x+K(t+s_n-r^*)+x_1(t+s_n-r^*-n_2T)).
\end{eqnarray*}
Then  one has 
 \begin{equation}\label{eq-vn} 
  w_n(t,x)\geq \Psi_n^-(t,x)+\Psi_n^+(t,x)-1-\tilde{q}(t+s_n-r^*-n_2T)
 \end{equation}
 for $t+s_n-r^*\geq (n_0+n_2)T$ and 
 $$x\in[-\underline{h}(t+s_n-r^*-n_2T)-K(t+s_n-r),\underline{h}(t+s_n-r^*-n_2T)-K(t+s_n-r^*)].$$ 
 Since $x_1(t)$ is bounded and 
 \begin{equation*}
  r_n\to r^*  ~\text{ and }  \underline{h}(t+s_n-r^*-n_2T)+K(t+s_n-r^*)\to \infty ~~\text{ as }n\to\infty,
 \end{equation*}
we have
 \begin{equation*}
  \lim_{n\to\infty}\Psi_n^+(t,x)=1 \text{ locally uniformly in } \mathbb{R}\times\mathbb{R}.
 \end{equation*}
Moreover, 
\begin{align*}
  &\underline{h}(t+s_n-r^*-n_2T)-K(t+s_n-r^*)\\
  &\geq \int_{0}^{t+s_n-r^*-n_2T}k(s)ds-\beta-\int_{0}^{t+s_n-r^*}k(s)ds+2L\\
&=-\int_{t+s_n-r^*-n_2T}^{t+s_n-r^*}k(s)ds-\beta+2L\\
&= -n_2T\bar{k} -\beta+2L=:Z_0~ \text{ for }t+s_n-r^*\geq (n_0+n_1)T.
\end{align*}
Since $x_1(t)\to 0$ as $t\to\infty$, letting $n\to\infty$ in \eqref{eq-vn} and using the monotonicity of $\Psi$, we see from  Lemma \ref{lem-lim} that 
\begin{equation}\label{eq-vphi}
  W(t,x)\geq \Psi(t,Z_0-x) \text{ for }(t,x)\in\mathbb{R}\times(-\infty,Z_0].
\end{equation}
Define
\begin{equation*}
  Z^*:=\sup\{Z\in\mathbb{R}:W(t,x)\geq \Psi(t,Z-x) \text{ for }(t,x)\in\mathbb{R}\times(-\infty,Z]\}.
\end{equation*}
Due to  \eqref{eq-vphi}, $W(t,I(t))=0$ and $\Psi(\cdot,x)>0$ for $x>0$, we see that $Z^*\leq I(r^*)$ is well-defined, and 
\begin{equation}\label{eq-vpr}\begin{cases}
  W(t,x)\geq \Psi(t,Z^*-x) \text{ for }(t,x)\in\mathbb{R}\times(-\infty,Z^*],\\
  \min_{t\in\mathbb R}I(t)=I (r^*)\geq  Z^*.
\end{cases}
\end{equation}

\begin{lemma}\label{lemm-gr}
$Z^*=I(r^*)$.
\end{lemma}
\begin{proof}
We proceed with a contradiction argument; assuming that $Z^*<I(r^*)$, we are going to derive a contradiction  in four steps.

\textbf{ Step 1.} We show that 
\begin{equation}\label{eq-v>p}
  W(t,x)> \Psi(t,Z^*-x) ~~\text{ for } (t,x)\in\mathbb R \times(-\infty,Z^*].
\end{equation}
We already know from \eqref{eq-vpr} that
\begin{equation*}
  W(t,x)\geq \Psi(t,Z^*-x)>0\text{ for }(t,x)\in\mathbb{R}\times(-\infty,Z^*).
\end{equation*}
As $W\geq 0$, the  parabolic strong maximum principle and \eqref{eq-V} then imply
\begin{equation}\label{eq-wpo}
   W(t,x)>0\text{ for }t\in\mathbb{R}, x\in(-\infty,I(t)).
\end{equation}
So it follows from $Z^*<I(r^*)\leq I(t)$ that $W(t,Z^*)>0=\Psi(t,0)$. Hence, we can apply the parabolic strong maximum principle to compare $W(t,x)$ and $\Psi(r,Z^*-x)$ over $\mathbb R\times (-\infty,Z^*)$ to  obtain \eqref{eq-v>p}.

 \textbf{ Step 2.} Given any $y<Z^*$,  we show that
 \begin{equation}\label{eq-chi}
  \varpi  (y):=\inf_{(t,x)\in\mathbb{R}\times[y,Z^*]}[W(t,x)-\Psi(t,Z^*-x)]>0.
 \end{equation}
By Step 1, $\varpi(y)\geq 0$. Suppose by contradiction that \eqref{eq-chi} does not hold. Then in view of \eqref{eq-v>p} there exist   $y_0\in(-\infty,Z^*)$, $|t_n|\to\infty$ and  $x_n\in [y_0,Z^*]$  such that 
\begin{equation*}
  \varpi(y_0)=\lim_{n\to\infty}[W(t_n,x_n)-\Psi(t_n,Z^*-x_n)]=0.
\end{equation*}
Write $t_n=\tilde{l}_nT+\tilde{r}_n$ with $\tilde{l}_n\in\mathbb N$ and $\tilde{r}_n\in[0,T)$. Passing to a subsequence, we may assume that $\tilde{r}_n\to\tilde{r}\in[0,T]$ and $x_n\to x_0^*\in[y_0,Z^*]$. Denote
\begin{equation*}
  (W_n(t,x), I_n(t)):=(W(t+t_n-\tilde{r},x+x_n),I(t+t_n-\tilde{r})-x_n).
\end{equation*} 
 Arguing as in the proof of Lemma \ref{lem-lim}, we can show that, by passing to a subsequence,
 \begin{equation*}
  (W_n(t,x), I_n(t))\to (W^\infty,I^\infty) \text{ in }C_{loc}^{(1+\alpha)/2,1+\alpha}(\tilde{\Sigma})\times C_{loc}^{1+\alpha/2}(\mathbb{R}),
 \end{equation*}
 where $\tilde{\Sigma}:=\{(t,x):t\in\mathbb{R},x<I^\infty(t)\}$ and $(W^\infty,I^\infty)$ solves
\begin{equation}\label{eq-Vinf}\begin{cases}
  W^\infty_t-dW^\infty_{xx}-k(t)W^\infty_x=f(t,W^\infty),~ W^\infty>0, &\text{ for }(t,x)\in\tilde{\Sigma},\\
  W^\infty(t,I^\infty(t))=0 &\text{ for }t\in\mathbb R
\end{cases}
\end{equation}
and 
\begin{equation}\label{eq-Vin}
  W^\infty(\tilde{r},0)=\Psi(\tilde{r}, Z^*-x_0^*).
\end{equation}
Due to $I(t)\geq I(r^*)$ and \eqref{eq-v>p}, we have
\begin{equation*}
  \begin{cases}
    I^\infty(t)+x_0^*\geq I(\tilde{r})\geq I(r^*)>Z^*,\\
     W^\infty(t,x)\geq \Psi(t,Z^*-x_0^*-x)~~ \text{ for }(t,x)\in\mathbb{R} \times (-\infty, Z^*-x_0^*].
  \end{cases}
\end{equation*} 
On the other hand, it is clear that $\Psi(t,Z^*-x_0^*-x)$ satisfies \eqref{eq-Vinf} with $I^\infty(t)$ replaced by $Z^*-x_0^*$. Since $I^\infty(t)>Z^*-x_0^*>0$, we have
\begin{equation*}
  W^\infty(t,Z^*-x_0^*)>0=\Psi(t,0) ~~\text{ for }t\in\mathbb{R}.
\end{equation*} 
Applying the parabolic strong maximum principle to compare $W^\infty(t,x)$ and $\Psi(t,Z^*-x_0^*-x)$ over $(-\infty,\tilde{r}]\times (-\infty,Z^*-x_0^*)$, we obtain
\begin{equation*}
  W^\infty(t,x) >\Psi(t,Z^*-x_0^*-x) \text{ for }t\leq \tilde{r},\ x<Z^*-x_0^*,
\end{equation*}
which contradicts \eqref{eq-Vin}. Therefore, \eqref{eq-chi} holds for all $y< Z^*$.

\textbf{Step 3.} We show that there exists $Z_1<Z^*$ such that 
\begin{center}
  $W(t,x)\geq \Psi(t,Z^*-x+\epsilon)$ in $\mathbb{R}\times(-\infty,Z_1]$ for all small $\epsilon>0$.
\end{center} 

 Let
 \begin{equation*}
  \sigma_0:=-\frac{1}{T}\int_0^Tf_u(s,1)ds.
 \end{equation*}
  By \eqref{smooth}, there exists $\epsilon_0>0$ small enough such that  
\begin{equation}\label{eq-f_u}
  |f_u(t,u)-f_u(t,1)|<\sigma_0 ~~\text{ for } t\in\mathbb R, u\in[1-\epsilon_0,1].
\end{equation} 
Due to $\Psi(t,x) \to\infty$ as $x\to\infty$ uniformly in $t\in\mathbb R$,  we can find $Z_1=Z_1(\epsilon_0)<Z^*$ with $Z^*-Z_1$ large enough such that 
\begin{equation}\label{eq-pl1}
  \Psi(t,Z^*-x)\geq 1-\epsilon_0 ~~\text{ for }t\in\mathbb{R},\ x\leq Z_1.
\end{equation}
By continuity of $\Psi$, there exists  $\epsilon\in(0,\epsilon_0)$ small such that
\begin{equation}\label{eq-pchi}
  \Psi(t,Z^*-Z_1+\epsilon)\leq \Psi(t,Z^*-Z_1)+\varpi(Z_1)~ \text{ for }t\in\mathbb R
\end{equation}
with $\varpi$ given by \eqref{eq-chi}.
Fix  such an $\epsilon$ and some $Z<Z_1$,  then consider  the problem
\begin{equation}\label{eq-wR}
  \begin{cases}
    v_t-dv_{xx}-k(t)v_x=f(t,v), &t>0, x\in(Z,Z_1),\\
    v(t,Z)=1, ~v(t,Z_1)=\Psi(t,Z^*-Z_1+\epsilon), &t>0,\\
    v(0,x)=1-\epsilon_0, &x\in[Z,Z_1].
  \end{cases}
\end{equation}
Noticing that $1$ and $1-\epsilon_0$ are a pair of upper and lower solutions of \eqref{eq-wR}, we can apply the standard upper and lower solution arguments to conclude that \eqref{eq-wR} admits a solution 
$$v_Z\in C_{loc}^{1+\alpha/2,2+\alpha}((0,\infty)\times[Z,Z_1])\cap C_{loc}^{1+\alpha/2,2+\alpha}([0,\infty)\times(Z,Z_1))$$
 satisfying
\begin{equation}\label{eq-1pw}
 1-\epsilon_0<  v_Z(t,x)< 1 ~~\text{ for }t>0,\ x\in(Z,Z_1).
\end{equation}
Moreover, it is easily seen that $v_Z$ increases in $Z$, which implies that the pointwise limit
\begin{equation*}
   v_\infty(t,x):=\lim_{Z\to -\infty}v_Z(t,x)
 \end{equation*}
 is well-defined for $t>0,\ x\in(-\infty,Z_1]$, and  $1-\epsilon_0\leq v_\infty\leq 1$.  By the standard parabolic estimates, this convergence also holds in $C_{loc}^{1,2}([0,\infty)\times(-\infty,Z_1))\cap C_{loc}^{1,2}((0,\infty)\times(-\infty,Z_1])$,  and thus $v_\infty$ solves \eqref{eq-wR} in $[0,\infty) \times(-\infty,Z_1]$  except for the first equation in the second line. Because  $f$  is $T$-periodic in $t$ and $1-\epsilon_0$ is a  lower solution, we see that $v_\infty(t+mT,x)$ is nondecreasing  in $m\in\mathbb R$. Therefore, the function
\begin{equation*}
  \tilde{v}(t,x):=\lim_{m\to\infty}v_\infty(t+mT,x)
\end{equation*}
is well-defined and  $T$-periodic  in $t$.   Similar as before,  we see  that  $\tilde{v}$ solves
\begin{equation}\label{eq-wtild}
  \begin{cases}
    \tilde{v}_t-d\tilde{v}_{xx}-k(t)\tilde{v}_x=f(t,\tilde{v}), ~1-\epsilon_0\leq \tilde{v}\leq 1, &t\in\mathbb R, x\in(-\infty,Z_1),\\
    \tilde{v}(t,Z_1)=\Psi(t,Z^*-Z_1+\epsilon), &t\in\mathbb R,\\
    \tilde{v}(t,x)=\tilde{v}(t+T,x), &t\in\mathbb R, x\in(-\infty,Z_1].
  \end{cases}
\end{equation}
A simple limit argument implies $\tilde{v}(t,-\infty)=1$. 
\iffalse For completeness, we provide a detailed proof. Let $\{x_n\}$  be any sequence such that $x_n\to-\infty$ as $n\to\infty$.  Denote 
\[\tilde{v}_n(t,x):=\tilde{v}(t,x+x_n) \text{ for } t\in\mathbb{R}, x\in(-\infty, Z_1-x_n],\]
then $\tilde{v}_n(t,x)$ solves \eqref{eq-wtild} with $Z_1$ replaced by $Z_1-x_n$. By the standard parabolic estimates, we can find a function $\hat{v}$ such that passing to a subsequence,
\[\tilde{v}_n(t,x)\to\hat{v}(t,x) \text{  as } n\to\infty \text{ locally uniformly in } C^{1, 2}(\mathbb{R}\times\mathbb{R}),\]
where $\hat{v}(t,x)$ satisfies 
\[ \begin{cases}
  \hat{v}_t-d\hat{v}_{xx}-k(t)\hat{v}_x=f(t,\hat{v}), ~1-\epsilon_0\leq \hat{v}\leq 1, &t\in\mathbb R, x\in\mathbb{R},\\
  \hat{v}(t,x)=\hat{v}(t+T,x), &t\in\mathbb R, x\in\mathbb{R}.
\end{cases}
\]
It follows from the usual comparison principle that $\hat{v}(t,x)\geq p(t)$ with $p(t)$ solving
\[p_t=f(t,p),~t>0;~~~p(0)=1-\epsilon_0.\]
Since $p(t+mT)\to 1$ locally uniformly in $t\in\mathbb{R}$ as the integer $m\to\infty$, we have
\[ 1= \lim_{m\to\infty}p(t+mT)\leq \lim_{m\to\infty}\hat{v}(t+mT,x)=\hat{v}(t,x)\leq 1 \text{ locally uniformly in }(t,x)\in\mathbb{R}\times\mathbb R,\]
which implies that 
\[\tilde{v}(t,x+x_n)\to 1 \text{ as }n\to\infty  \text{  in } C^{1, 2}_{loc}(\mathbb{R}\times\mathbb{R}).\]
Since  the sequence $\{x_n\}$ is arbitrary, one obtains that $\tilde{v}(t,-\infty)=1$. 
\fi

Now, we define
$$\Phi(t,x):=\Psi(t,Z^*-x+\epsilon) ~\text{ for }t\in\mathbb R,\ x\leq Z_1.$$
Then $\Phi$ also solves \eqref{eq-wtild} and  $\Phi(t,x)\geq 1-\epsilon_0$ for $x\leq Z_1$ by \eqref{eq-pl1}. Applying the  comparison principle to $\Phi$ and $v_\infty$, we obtain
\begin{equation*}
  \Phi(t+mT,x)\geq  v_\infty(t+mT,x)~~\text{ for } t>-mT,\  x\in(-\infty,Z_1],\  m\in\mathbb N.
\end{equation*}
Passing to the limit  $m\to\infty$, we obtain
\begin{equation}\label{eq-wp1}
  \Phi(t,x)\geq \tilde{v}(t,x)\geq 1-\epsilon_0~~\text{ for }t\in\mathbb{R},\ x\in(-\infty,Z_1].
\end{equation}
We claim that 
\begin{equation*}
  \Phi(t,x)\equiv \tilde{v}(t,x)~~ \text{ for }t\in\mathbb{R},\ x\in(-\infty,Z_1].
\end{equation*}
Define
\begin{equation*}
  U(t,x):=e^{-\sigma_0t-\int_0^tf_u(s,1)ds}[\Phi(t,x)-\tilde{v}(t,x)].
\end{equation*}
A direct calculation shows that 
\begin{equation}\label{eq-U}\begin{cases}
  U_t-dU_{xx}-k(t)U_x+[f_u(t,1)-f_u(t,\theta(t,x))+\sigma_0]U=0,&t\in\mathbb{R},\ x\in(-\infty,Z_1),\\
  U(t,-\infty)=U(t,Z_1)=0,&t\in\mathbb{R},\\
  U(t,x)=U(t+T,x), &t\in\mathbb{R},x\in(-\infty, Z_1],
\end{cases}
\end{equation}
where $\theta(t,x)\in[\tilde{v}(t,x),\Phi(t,x)]\subset[1-\epsilon_0,1]$.
If the claim does not hold, then there exists $(t_0,x_0)\in\mathbb{R}\times(-\infty,Z_1)$ such that 
$$U(t_0,x_0)=\max_{(t,x)\in\mathbb{R}\times(-\infty,Z_1)}U(t,x)>0.$$
By \eqref{eq-f_u} and the first equation of \eqref{eq-U}, we have
\begin{equation*}
  0<\big[U_t-dU_{xx}-k(t)U_x+[f_u(t,1)-f_u(t,\theta(t,x))+\sigma_0]U\big]|_{(t_0,x_0)}=0.
\end{equation*}
This  contradiction shows the claim holds.

From \eqref{eq-v>p},\eqref{eq-chi} and \eqref{eq-pchi}, we see that
\begin{eqnarray*}
  \begin{cases}
     W(t,x)> \Psi(t,Z^*-x)\geq 1-\epsilon_0 &\text{ for } (t,x)\in\mathbb R \times(-\infty,Z_1],\\
  W(t,Z_1)\geq \Psi(t,Z^*-Z_1)+\varpi (Z_1)\geq  \Psi(t,Z^*-Z_1+\epsilon) &\text{ for }t\in\mathbb R. 
  \end{cases}
\end{eqnarray*}
 The usual comparison principle then yields that  
\[ W(t,x)\geq v_\infty(t+mT,x)~\text{ for }t>-mT,\ x\in(-\infty,Z_1],\  m\in\mathbb N.\]
Passing to the limit  $m\to\infty$, we obtain the desired inequality
\begin{equation}\label{eq-vw}
  W(t,x)\geq \tilde{v}(t,x)=\Psi(t,Z^*-x+\epsilon) ~\text{ for }t\in\mathbb R,\ x\in(-\infty,Z_1].
\end{equation}

\textbf{Step 4.} Completion of the proof.

By the equations satisfied by $W$, we can easily deduce from the parabolic $L^p$  estimates and the Sobolev embedding theorem that for some $C>0$,
\[
|W_x(t,x)|\leq C \mbox{ for } t\in\mathbb R,\ Z^*\leq x\leq I(t).
\]
(This can also be obtained by using \eqref{L^p-est}.)
Denote
\begin{equation*}
\varrho :=\varpi(Z_1)>0.
\end{equation*}
By the above bound for $W_x$ and the continuity and periodicity of $\Psi$, we see that for sufficiently small $\epsilon_1\in(0,\varrho ]$,
\begin{eqnarray*}\begin{cases}
 W(t,x)>W(t,Z^*)-\varrho/2 &\text{ for }t\in\mathbb R,\ x\in[Z^*,Z^*+\epsilon_1],\\
 \Psi(t,Z^*-x+\epsilon_1)\leq \Psi(t,Z^*-x)+\varrho/2 &\text{ for }t\in\mathbb R,\ x\in[Z_1, Z^*].
\end{cases}
\end{eqnarray*}
Recalling 
\begin{equation*}
W(t,x)-\Psi(t,Z^*-x)\geq \varpi (Z_1)=\varrho  \text{ for }(t,x)\in\mathbb{R}\times[Z_1,Z^*],
\end{equation*}
we readily see that
\begin{equation*}
  W(t,x)-\Psi(t,Z^*-x+\epsilon_1)\geq \Psi(t,Z^*-x)+\varrho - \Psi(t,Z^*-x)-\varrho/2>0 \text{ for }t\in\mathbb{R},\ x\in[Z_1,Z^*].
\end{equation*}
In addition, by decreasing $\epsilon_1$ if necessary, we also have
\begin{equation*}
  \Psi(t,Z^*-x+\epsilon_1)<\Psi(t,0)+\varrho/2 ~~\text{ for } x\in[Z^*,Z^*+\epsilon_1],\ t\in\mathbb R.
\end{equation*}
It  follows  that 
\begin{equation*}
  W(t,x)-\Psi(t,Z^*-x+\epsilon_1)> W(t,Z^*)-\varrho /2 - \Psi(t,0)-\varrho/2\geq0 \text{ for }  t\in\mathbb{R},\ x\in[Z^*,Z^*+\epsilon_1].
\end{equation*}
Combining these inequalities with \eqref{eq-vw} we obtain, for all small $\epsilon_1\in(0,\epsilon)$,
\begin{equation*}
  W(t,x)-\Psi(t,Z^*-x+\epsilon_1)\geq0 \text{ for } t\in\mathbb{R},\ x\in(-\infty,Z^*+\epsilon_1).
\end{equation*}
This contradicts the definition of $Z^*$. Consequently, $Z^*=I(r^*)$ and the lemma is  proved.
\end{proof}

\begin{proposition}\label{prop-wpsi}
  $W(t,x)\equiv \Psi(t,Z^*-x)$ and $I(t)\equiv Z^*$.
\end{proposition}

\begin{proof}
In view of Lemma \ref{lemm-gr} and \eqref{eq-vpr}, we see that 
  \begin{equation*}
    \begin{cases}
      Z^*=I(r^*)=\min_{t\in\mathbb R} I(t),\\
      W(t,x)\geq \Psi(t,Z^*-x) \text{ for }t\in\mathbb{R},\ x\leq Z^*.
    \end{cases}
  \end{equation*}
Since
  \begin{equation*}
    W(r^*,I(r^*))=0=\Psi(r^*,0),
  \end{equation*}
 it follows from the parabolic strong maximum principle and the Hopf boundary lemma that 
  \begin{equation*}
    \text{ either } W(t,x)\equiv \Psi(t,Z^*-x) \text{ for }x\leq Z^* \ \text{ or } \ W_x(r^*,I(r^*))<-\Psi_x(r^*,0).
  \end{equation*} 
On the other hand, \eqref{eq-V} and the minimality of $I(r^*)$ imply
  \begin{equation*}
   0= I'(r^*)=-\mu W_x(r^*,I(r^*))-k(r^*),
  \end{equation*}
which gives 
  \begin{equation*}
    W_x(r^*,I(r^*))=-\frac{1}{\mu} k(r^*)=-\Psi_x(r^*,0).
  \end{equation*}
Therefore, we necessarily have 
$$W(t,x)\equiv \Psi(t,Z^*-x) \text{ for }x\leq Z^*.$$ 
Due to $W(t,x)>0$ for $x<I(t)$ by \eqref{eq-wpo}, we further deduce $I(t)\equiv Z^*$.
\end{proof}
\subsection{Completion of the proof of Theorem \ref{thm-exact}}
The arguments below follow those in \cite{DMW25} with minor modifications. For completeness, we provide the details, which are divided into several steps.

\textbf{Step 1.} Let $\{s_n\}$ be the sequence chosen in Lemma \ref{lem-lim} (which is a subsequence of \eqref{sn}), with $ s_n=l_nT+r_n,\ l_n\in\mathbb N, \ r_n\in [0, T),\ r_n\to r^*$. We prove that 
\begin{equation*}
  \begin{cases}
  \lim_{n\to\infty}(h'(t+s_n-r^*)-k(t+s_n-r^*))=0 \text{ for }t\in\mathbb R,\\
  \lim_{n\to\infty}\sup_{x\in[0,h(s_n)]}|u(s_n,x)-\Psi(s_n,h(s_n)-x)|=0.\\
  \end{cases}
\end{equation*}

By Lemmas \ref{lem-lim} and \ref{lemm-gr} and Proposition \ref{prop-wpsi}, we have
\begin{equation*}
\begin{cases}
  H_n(t)=h(t+s_n-r^*)-\int_0^{t+s_n-r^*}k(s)ds+2L\to Z^* \text{ in } C_{loc}^{1+\alpha/2}(\mathbb R) \text{ as }n\to\infty,\\
  \lim_{n\to\infty}u(t+s_n-r^*,x+h(t+s_n-r^*))=\Psi(t,-x) \text{ in } C_{loc}^{(1+\alpha)/2,1+\alpha}(\mathbb{R}\times(-\infty,0]).
\end{cases}
\end{equation*}
It follows that 
\[h'(t+s_n-r^*)-k(t+s_n-r^*)\to 0 \text{ in } C_{loc}^{\alpha/2}(\mathbb R) \text{ as }n\to\infty,\]
and for any $R>0$,
\begin{equation}\label{eq-uphi}\begin{aligned}
  &\lim_{n\to\infty}\|u(s_n,\cdot)-\Psi(s_n,h(s_n)-\cdot)\|_{L^\infty( [h(s_n)-R,h(s_n)])}\\
  &=\lim_{n\to\infty}\|u(s_n,\cdot)-\Psi(r^*,h(s_n)-\cdot)\|_{L^\infty( [h(s_n)-R,h(s_n)])}=0.
\end{aligned}
\end{equation}

 By  Lemmas \ref{lemma-upp} and \ref{lemma-lower},  for any given $\epsilon>0$,  there exists  $N\in\mathbb{N}$ and  $R_1>0$ large enough so that 
\[1-\epsilon\leq  u(s_n,x)\leq 1+\epsilon ~\text{ for }x\in[0,h(s_n)-R_1], n\geq N.\]
Upon increasing $R_1$ if needed, we also  have
\[1-\epsilon\leq \Psi(s_n, h(s_n)-x)\leq 1+\epsilon \  \text{ for } n\in\mathbb N, x\in(-\infty,h(s_n)-R_1].\]
Therefore, 
\[ \|u(s_n,\cdot)- \Psi(s_n, h(s_n)-\cdot)\|_{L^\infty([0,h(s_n)-R_1])}\leq 2\epsilon \ \text{ for } n\geq N.\]
 This and \eqref{eq-uphi} imply
\begin{equation}\label{eq-limup}
  \lim_{n\to\infty}\|u(s_n,\cdot)- \Psi(s_n, h(s_n)-\cdot)\|_{L^\infty([0,h(s_n)])}=0.
\end{equation}

 \textbf{Step 2.} We show that 
\begin{equation}\label{eq-conv}
\lim_{t\to\infty}\left[h(t)-\int_0^tk(s)ds\right]=h^*:=Z^*-2L.
\end{equation}

By Step 1, we have \eqref{eq-limup} and 
\begin{equation}
  \lim_{n\to\infty}[h'(s_n)-k(s_n)]=0.
\end{equation}
  By the choice of $\{s_n\}$ we also have
 \[\lim_{n\to\infty}\left[h(s_n)-\int_0^{s_n}k(s)ds+2L\right]=\liminf_{t\to\infty}\left[h(t)-\int_0^{t}k(s)ds+2L\right]=Z^*.\]
Suppose by contradiction that \eqref{eq-conv} does not hold; then 
there exists a sequence $\{t_n\}$ such that $t_n\to \infty$ as $n\to\infty$ and 
\[\lim_{n\to\infty}\left[h(t_n)-\int_0^{t_n}k(s)ds\right]=\limsup_{t\to\infty}\left[h(t)-\int_0^{t}k(s)ds\right]=\check{h} >h^*.\]
In the following, we are going to derive a contradiction. We first adjust the parameters in the upper solution constructed in \eqref{eq-ups}  by utilizing \eqref{eq-limup}. We now take $M_0=\gamma:=(\check{h}-h^*)/9$, $0<a\leq(\check{h}-h^*)/9$ and $\tau_0:=s_n$ with $n\in\mathbb{N}$ to be determined. Checking the proof of Lemma \ref{lemma-upp},  we can choose  $\delta\in(0,m_a)$ sufficiently small so that the inequalities \eqref{eq-gq1}, \eqref{eq-gq2}, \eqref{eq-gq3} and \eqref{eq-gq4} hold, which lead to 
\begin{eqnarray*}
  \begin{cases}
    \bar{h}'(t)\geq -\mu \bar{u}_x(t,\bar{h}(t)) &\text{ for }t\geq s_n, \\
\bar{u}_t-d\bar{u}_{xx}-f(t,\bar{u})\geq 0 & \text{ for }t\geq s_n,\  x\in[g(t),\bar{h}(t)].
  \end{cases}
 \end{eqnarray*}
%Since $\bar{h}(s_n)-x+x_0(t)\geq M_0>0$ for $x\in[g(s_n),h(s_n)]$, one has
%$$\bar{\Psi} (s_n,\bar{h}(s_n)-x+x_0(t))=\Psi(s_n,\bar{h}(s_n)-x+x_0(t))\geq \Psi(s_n,M_0)>0 \text{ for }x\in[g(s_n),h(s_n)].$$
Moreover, by \eqref{eq-limup}, there exists $n_1\in\mathbb{N}$ large enough so that for all $n\geq n_1$,
\begin{eqnarray*}
  u(s_n,x)- \Psi(s_n, h(s_n)-x)<q_0  \text{ for }x\in[0,h(s_n)].
\end{eqnarray*}
Since  $\bar{h}(s_n)\geq h(s_n)+a$ and $x_1(s_n)\to 0$ as $n\to\infty$, we further deduce from the monotonicity of $\Psi$  that 
\begin{eqnarray}\label{eq-usn}
  u(s_n,x)< \Psi(s_n, \bar{h}(s_n)-x+x_0(s_n))+q_0=\bar{u}(s_n,x) \text{ for }x\in[0,h(s_n)],\ n\geq n_1.
\end{eqnarray}
In addition, it is obvious that 
\begin{equation*}
\Psi(s_n, \bar{h}(s_n)-x+x_0(s_n))+q_0\geq \Psi(s_n, h(s_n))+q_0\geq 1+\frac{q_0}{2}\geq u(s_n,x) \text{ for } x\in[g(s_n),0],\ n\geq n_1\gg 1.
\end{equation*} 
 This and \eqref{eq-usn} imply
\begin{equation*}
\bar{u}(s_n,x)\geq u(s_n,x) \text{ for } x\in[g(s_n),h(s_n)],\ n\geq n_1\gg 1.
\end{equation*} 
Combining these inequalities with the first and second lines of \eqref{eq-sum}, we see that   $(\bar{u},g, \bar{h})$ is an upper solution of \eqref{free-bound} for $t\geq \tau_0:=s_n$ with $n\geq n_1$. It then follows from Lemma \ref{lem-comp} that
\[h(t)\leq \bar{h}(t)=\int_{s_{n}}^tk(s)ds+\gamma(1-e^{-\alpha\sigma_0(t-s_{n})})+h(s_{n})+a+M_0 \text{ for }t\geq s_n. \]
Hence, for all large $k$ satisfying $t_k>s_n$, we have
\[h(t_k)-\int_0^{t_k}k(s)ds\leq h(s_{n})-\int_0^{s_{n}}k(s)ds+\gamma(1-e^{-\alpha\sigma_0(t_k-s_{n})})+a+M_0.\]
Passing to the limit $k\to\infty$ and then $n\to\infty$, we obtain
\[\check{h}-h^*\leq \gamma+a+M_0\leq (\check{h}-h^*)/3,\]
which is a contradiction. Thus \eqref{eq-conv} holds, as desired.

\textbf{Step 3.}  We show that $\lim_{t\to\infty} \sup_{x\in[0,h(t)]}|u(t,x)-\Psi(t,h(t)-x)|=0$.

Let $\{t_n\}$ be any sequence such that  $t_n\to\infty$ as $n\to\infty$. It follows from Step 2 that 
$$\lim_{n\to\infty}H(t_n)=\liminf_{t\to\infty}H(t).$$
Hence, we can apply  Lemma \ref{lem-lim} with $\{t_n\}$ in place of $\{s_n\}$ in \eqref{sn} to conclude that \eqref{eq-limup} holds with $\{s_n\}$ replaced by a subsequence of $\{t_n\}$. The arbitrariness of  $\{t_n\}$ then implies
$$\lim_{t\to\infty} \sup_{x\in[0,h(t)]}|u(t,x)-\Psi(t,h(t)-x)|=0.$$
In view of \eqref{u0(-x)}, 
 the proof of the theorem is now complete.\qed

\section*{Acknowledgments}
This research is supported by The Australian Research Council. %We thank the referees for carefully checking the paper and useful suggestions.

\section*{Data availability statement}
There is no data associated with this research.

\bibliographystyle{plain}

\end{document}